%% file: SINUM_IPM_Rev.tex
\documentclass[review,onefignum,onetabnum]{siamart171218}

\input{ex_shared}
\nolinenumbers

\myexternaldocument{ex_supplement}

\usepackage[utf8]{inputenc}
\usepackage[T1]{fontenc}
\usepackage{mathptmx,physics}
\usepackage{etoolbox}
\usepackage{comment,todonotes}
\usepackage{amsfonts,amssymb}
 \usepackage{bm}
\usepackage{graphicx}
\usepackage{color}

\usepackage{hyperref}
\usepackage{multirow}
\usepackage{pifont}

\usepackage{enumitem}
\usepackage{ulem}

\usepackage{graphicx,epstopdf} 
\usepackage[caption=false]{subfig} 

\newcommand{\bu}{{\bf u}}               

\newtheorem{thm}{Theorem}
\newtheorem{lem}[thm]{Lemma}
\newtheorem{pro}[thm]{Proposition}
\newtheorem{dfn}{Definition}

\newtheorem{rem}{Remark}
\newtheorem{asm}{\bf Assumption}

\newcommand{\wt}[1]{\widetilde{#1}}

\newcommand{\inpd}[2]{\left\langle #1, #2 \right\rangle}

\newcommand{\Real}{{\mathbb{R}}}

 \def\Lip{\operatorname{Lip}}

\def\argmin{\operatorname{argmin}}

\newcommand{\bv}[1]{\mathbf{#1}}

\begin{document}
\today
\maketitle


\begin{abstract}
Computing  saddle points with a prescribed  Morse index on potential energy surfaces is crucial for characterizing   transition states for noise-induced rare transition events in physics and chemistry. 
Many numerical algorithms for this type of saddle points are based on the eigenvector-following  idea and can be cast as an iterative minimization formulation (SINUM. Vol. 53, p.1786, 2015), but they may struggle with convergence issues and require good initial guesses.  To address this challenge, we discuss the differential game interpretation of this iterative minimization formulation and investigate the relationship between this game's Nash equilibrium and saddle points on the potential energy surface. Our main contribution  is that adding a proximal term, which   grows faster than quadratic, to the game's cost function can enhance the  stability and robustness.
This approach produces a robust Iterative Proximal Minimization (IPM) algorithm for saddle point computing. We show that the IPM algorithm surpasses the preceding methods in robustness without compromising the convergence rate or increasing computational expense. The algorithm's efficacy and robustness are showcased through a two-dimensional test problem, and the Allen-Cahn,  Cahn-Hilliard equation, underscoring its numerical robustness.
\end{abstract}



\section{Introduction}
Transition states bearing substantial physical significance have long attracted interest across disciplines, including physics, chemistry, biology, and material sciences  \cite{Energylanscapes}.
In the realm of computational chemistry, the transition state holds prominence on the potential energy surface related subjects, in addition to the multitude of local minimum points. 
These transition states act as bottlenecks on the most probable paths transitioning between different local wells   \cite{TPSChandler2002,String2002}. Typically, the transition state is characterized as a special type of saddle point with a Morse index of 1, defined as the critical point with a single unstable direction \cite{Erying,FW2012}.

In recent years, an extensive array of numerical methods have been proposed \cite{cerjan1981,Dimer1999, GAD2011, IMF2015} and developed \cite{KS2008,DuSIAM2012, Ren2013,SimGAD2018, IMA2015,zhang2016optimization} to effectively compute these index-1 saddle points. Most of these methods \cite{GAD2011,IMF2015,yin2019high} can be generalized to cases with a Morse index $\ge 1$. Furthermore, for applications involving multiple unstable solutions of certain nonlinear partial differential equations, computational methods based on mountain pass and local min-max methods  \cite{MPA1993,ZHOUJXSISC2001,ZHOUJXSISC2012} have been developed.


A notable category of algorithms for the location of saddle points of a given index is based on the idea of utilizing the Min-Mode direction \cite{Crippen1971,Dimer1999, GAD2011,DuSIAM2012, IMF2015, zhang2016optimization,SamantaGAD}. The Iterative Minimization Formulation (IMF)  \cite{IMF2015} establishes a rigorous mathematical model for this min-mode concept, fostering progression, generalization, and analysis of a range of IMF-based algorithms  \cite{LOR2013,IMA2015,ProjIMF,MsGAD2017,SimGAD2018, Ortner2016dimercycling}.
Despite the quadratic convergence rate, its convergence is localized and heavily dependent on the quality of the initial conditions.
In practical scenarios, this challenge can be substantially alleviated by adopting an adaptive inexact solver  \cite{IMA2015} for the sub-problems  within the IMF. Generally, this trade-off between efficiency and robustness is effective for the application of these algorithms. However, it could require fine-tuning of parameters and initialization strategies. Furthermore, the continuous model of the IMF, known as Gentlest Ascent Dynamics (GAD)  \cite{GAD2011}, which utilizes a single-step gradient descent to tackle each minimization sub-problem, has been empirically found to exhibit greater robustness than the conventional IMF \cite{IMA2015,GAD-DFT2015}, albeit at a linear convergence rate.
 
These research findings underscore the necessity for an in-depth probe into the root causes of numerical divergence issues. We posit that the lack of convexity in the sub-problem of optimizing an auxiliary function within the IMF contributes, at least in part, to this issue. To tackle this practical challenge, we embrace an innovative approach grounded in game theory and its correlation to the saddle point.
The IMF utilizes a position variable and an orientation vector, both of which independently minimize their own objective functions. This structure echoes a differential game model  \cite{2018mechanics,gemp2020,DGM2019, omidshafiei2020navigating}, where the Nash equilibrium serves as a pivotal concept in game theory \cite{nash1950}. We first aim to explicitly scrutinize the link between current IMF algorithms and differential game models. It will become evident that to ensure a differential game is well-defined and compatible with the IMF, an enhancement of the existing auxiliary function used in the original IMF is required.

Our primary approach hinges on the incorporation of a proximal function as the penalty to the pre-existing auxiliary function, thereby assuring its strict convexity.
Similar proximal point methods and related techniques have been effectively employed in the optimization literature. For example,  \cite{kaplan1998proximal} provides a comprehensive study of proximal point methods;  \cite{bolte2014proximal} presents a proximal alternating linearized minimization method that provides a robust solution for non-convex and non-smooth optimization problems;  \cite{latafat2023adabim} introduces AdaBiM to effectively solve structured convex bi-level optimization problems; and \cite{grimmer2023landscape} provides an in-depth exploration of the proximal point method within the context of non-convex non-concave min-max optimization.


Capitalizing on this innovative proximal technique, we devise a new approach, called the iterative proximal minimization (IPM) method.  We first establish the equivalence between the saddle point of the potential function, the fixed point of the proposed iterative scheme, and the Nash equilibrium of the differential game.
We  reveal  that, unlike the squared $L_2$ norm in traditional proximal point algorithms, the proximal function must exhibit a growth rate faster than a quadratic function. This method ensures that each sub-problem possesses a well-defined minimizer, irrespective of the initial value. Remarkably, the method converges swiftly even when the initial estimate substantially deviates from the true saddle point.

The paper is organized as follows. Section \ref{background} discusses the background of the saddle point, the Morse index,  and the review of the original IMF.
In Section \ref{IPM}, we develop   the Iterative Proximal Minimization scheme, and discuss the equivalence of the Nash equilibrium and the saddle point by introducing the game model, and we present the Iterative Proximal Minimization Algorithm for this new method.  
 Section \ref{Num_ex} are for the numerical tests. Section \ref{con} is the conclusion.

\section{Background}\label{background}
	\subsection{Saddle Point and Morse Index} The potential function $V$ is a ${C}^3$ function on $\mathbb{R}^d$.
	The saddle point $x^*$ of a potential function $V(x)$ is a critical point at which the partial derivatives of a function $V(x)$ are zero but are not an extremum. The Morse index of a critical point of a smooth function $V(x)$ on a manifold is the negative inertia index of the Hessian matrix of the function $V(x)$.   $\norm{x}$ is the Euclidean norm of a point $x\in \Real^d$. 
  $\norm{x}_p$ is the  $L^p$ norm.
 Since $(\Real^d, L^2)$ is flat, the tangent space is  also $\Real^d$ and $\norm{\bf v}$  
 is also the Euclidean norm of the vector $\bv v$.
    $\mathbf{I}$ denotes the identical matrix of $\mathbb{R}^{d\times d}$.
       $\mathcal{B}_\epsilon (x) := \{y\in \mathbb{R}^d: \|x-y\| \le \epsilon\}.$
 
{\bf Notations:}
	 \begin{enumerate}
	     \item $\lambda_i(x)\in\mathbb{R}$ for any $x\in\mathbb{R}^d$ is the  $i$-th (in the ascending order) eigenvalue of the Hessian matrix $H(x) = \nabla^2 V(x)$;
	     \item $\bv v_i(x)$ is an eigenvector corresponding to the $i$-th eigenvalue $\lambda_i(x)$ of the Hessian matrix $H(x)$ and the projection matrix
        $\Pi_1(x) := \bv v_1(x)\bv v_1(x)^\top$.
	     \item $\mathcal{S}_1$ is the collections of all {\it non-degenerate index-1 saddle points} $x^*$ of $V$, defined by
	     $$
	     \nabla V(x^*) = \mathbf{0}, \quad \text{ and } \lambda_1(x^*)< 0 <\lambda_2(x^*)\leq\cdots\leq \lambda_d(x^*);
	     $$
	     \item $\Omega_{1}$ is a subset of $\mathbb{R}^d$ called index-1 region, defined as
	     $$
	     \Omega_{1}:=\{x\in \mathbb{R}^d:  \lambda_1(x)<0<\lambda_2(x)\};
	     $$
	     \item $\lambda_{\min}(\cdot),\lambda_{\max}(\cdot):\mathbb{R}^{d\times d}\rightarrow\mathbb{R}$ is the smallest and the largest eigenvalue of a (symmetric) matrix respectively, so by the Courant-Fisher theorem, 
        \begin{equation}\label{eig_mm}
            \lambda_{\min}(A) = \min_{\|z\|=1}z^\top A z,\qquad \lambda_{\max}(A) = \max_{\|z\|=1}z^\top A z;
        \end{equation}
      
	 \end{enumerate}

	\subsection{Review of  Iterative Minimization Formulation}

     The iterative minimization algorithm to search the saddle point of index-1      \cite{IMA2015} is given by
    \begin{equation}\label{eq:imf}
    \left\{\begin{array}{l}
    \bv u^{t}=\argmin_{\|\tilde{\bv u}\|=1} \tilde{\bv u}^{\top} H\left(x^{t}\right) \tilde{\bv u}, \\
    x^{t+1}=\argmin_{y\in\mathcal{U}(x^t)}W(y;x^{t},\bv u^{t}),
    \end{array}\right.
    \end{equation}
    where $\mathcal{U}(x)\subset\mathbb{R}^d$ is a local neighborhood of $x$, $H(x) := \nabla^2 V(x)$ and the auxiliary function 
    \begin{equation} \label{W}
            W(y;x, \bv u) := (1-\alpha)V(y) + \alpha V\left(y- \bv u  \bv u^\top(y-x)\right)-\beta V\left( x +  \bv u  \bv u^\top(y-x)\right ).
    \end{equation}
The constant  parameters $\alpha$ and $\beta$ should satisfy $\alpha+\beta>1$.    
    The fixed point $(x^*,\bv u^*)$ of the above iterative scheme  satisfy
    \begin{equation} \label{IMF}
    \left\{\begin{array}{l}
    \bv u^*=\argmin_{\|\tilde{\bv u}\|=1} \tilde{\bv u}^{\top} H\left(x^{*}\right) \tilde{\bv u} =\bv v_1(x^*),\\
    x^{*}=\argmin_{y\in\mathcal{U}(x^*)}W(y;x^*,\bv u^*).
    \end{array}\right.
    \end{equation}
     In addition, suppose that 
    $
    H(x^*) =\bv v\Lambda \bv v^\top,
    $
    where $\Lambda = \operatorname{diag}(\lambda_i)$ and ${\bv v}=[\bv v_1,\ldots,\bv v_d]$.
    Then, a necessary and sufficient condition for
    \[ 
    x^{*}=\underset{y\in\mathcal{U}(x^*)}{\argmin}~ W(y;x^*,\bv v_1(x^*))
    \]
    is that 
    \begin{equation}\label{eq:lam}
    \lambda_{1}(x^*)<0<\lambda_{2}(x^*)\leq\cdots\leq\lambda_{d}(x^*).    
    \end{equation}

    The main result about the iterative scheme \eqref{eq:imf} is quoted here from the IMF paper \cite{IMF2015}.
    Suppose that $\alpha,\beta\in\mathbb{R}$ satisfying $\alpha+\beta>1$,  and $W$ is defined by \eqref{eq:imf}. 
    Suppose that $x^*$ is a non-degenerate index-1 saddle point of the function $V(x)$. Then the following statements are true.
    \begin{enumerate}
        \item $x^*$ is a local minimizer of $W(y;x^*, \bv v_1(x^*) )$;
        \item There exists a neighborhood $\mathcal{U}(x^*)$ of $x^*$ such that for any $x\in\mathcal{U}(x^*), W(y;x, \bv v_1(x) )$ is strictly convex in $y\in\mathcal{U}(x^*)$ and thus has a unique minimum in $\mathcal{U}(x^*)$;
        \item Define the mapping $\Psi:x\in\mathcal{U}(x^*)\rightarrow\Psi(x)\in\mathcal{U}(x^*)$ where $\Psi(x)$ is the unique local minimizer of $W$ in $\mathcal{U}(x^*)$. Further, assume that $\mathcal{U}$ contains no other stationary points of $V$ except $x^*$. Then the mapping has only one fixed point $x^*$.
        \item If $V\in C^4$, then the iterative scheme  \eqref{eq:imf}  has exactly the second order
(local) convergence rate.
    \end{enumerate}

\subsection{Game Theory Interpretation}\label{Game_theory}

Game theory delves into the mathematical modeling of strategic interactions among agents. Each agent performs actions and incurs costs, which are functions of the actions executed by all game participants.

A {\it game} of $n$-agents, denoted as $G_n = (P,A,C)$, comprises $n$ agents represented by the set $P = \{1,\cdots,n\}$. The action space of these agents is denoted by $A = A_1\times\cdots\times A_n$. The set of reward (or cost) functions is represented by $C = {C_i(a_1,\ldots,a_n):A\rightarrow\mathbb{R}, i\in P}$. Each agent $i$ can perform actions within his own action space, aiming to minimize his cost $C_i$.

A fundamental concept in game theory is the {\it Nash equilibrium}  \cite{nash1950}, which delineates an action profile where no agent can reduce his own cost by altering his individual action, assuming that all other agents maintain their actions. A Nash equilibrium may represent a pure action profile or a mixed one, corresponding to each agent's specific action or a probability distribution over the action space, respectively.

\begin{dfn}[Pure Nash equilibrium]
\label{def-Nash}
For a game of $n$-agents  $G_n = \left(P,A,C\right)$,\\
$P = \{1,2,...,n\}$, $A = A_1\times\cdots\times A_n$, $C = \{C_i:A\rightarrow\mathbb{R}, i\in P\}$ denoting 
the set of agents, action space, and cost functions, respectively, an action profile $a^* = (a^*_1,...,a^*_n)$ is said to be a {pure Nash equilibrium} if  
\begin{equation}
C_k(a^*_1,...,a^*_{k-1},a^*_k,a^*_{k+1},...,a^*_n)\leq C_k(a^*_1,...,a^*_{k-1},a_k,a^*_{k+1},...,a^*_n), \quad  \forall a_k\in A_k,
\forall k\in P.
\end{equation}
Furthermore, if the inequality above holds strictly for all $a_k\in A_k,\text{ and } k\in P$, then the action profile $a^*$ is called a { strict pure Nash equilibrium}.
\end{dfn}

We are interested in constructing a game theory model in which the fixed point in the iterative minimization scheme \eqref{IMF} corresponds to the pure Nash equilibrium of the associated game, with each agent aiming to minimize its costs.
To achieve this, we introduce a new agent with action $y\in\mathcal{U}(x)$, in addition to the two agents with action $(x,\bv u)\in\mathbb{R}^d\times\mathbb{S}^{d-1}$ in the auxiliary function $W$ in the IMF \eqref{eq:imf}. 
We propose a simple penalty cost $\frac{1}{2}|x-y|^2$ to ensure the synchronization between agents $x$ and $y$. The condition in \eqref{IMF} then equates to Equation \eqref{eq:imf_fp}.

\begin{equation} \label{eq:imf_fp}
    \left\{\begin{array}{l}
    \bv u^*=\argmin_{\|\tilde{\bv u}\|=1} \tilde{\bv u}^{\top} H\left(x^{*}\right) \tilde{\bv u}, \\
    y^{*}=\argmin_{y\in\mathcal{U}(x^*)}W(y;x^*,\bv u^*),\\
    x^* = \argmin_{x\in\mathbb{R}^d}\frac{1}{2}\|x-y^*\|^2.\\
    \end{array}\right.
\end{equation}

It becomes evident that each of $x^*, y^*$, and $\bv u^*$ minimizes their individual cost function of $(x,y,\bv u)$. This condition \eqref{eq:imf_fp} naturally leads us to interpret the fixed points of the iterative scheme as the Nash equilibrium of the corresponding game. However, the local neighborhood $\mathcal{U}(x)$, essentially the feasible set for $y$ to ensure the well-defined minimization of $W$, depends on the action of agent $x$.

In practice, the restriction of this local neighbor $\mathcal{U}(x)$ is resolved using the parameter $x$ as an initial guess to minimize $W$ over $y$  \cite{GAD2011, IMA2015}. But in the classic game theory setup, the action space of each agent $A_i$ should be independent of the actions taken by other agents \cite{osborne1994}. Thus, we cannot directly formulate a game whose Nash equilibrium is characterized by \eqref{eq:imf_fp}.
To eliminate this ambiguity of local constraint from $\mathcal{U}(x)$, we penalize the agent   $y$ when $y\notin\mathcal{U}(x)$ by a modified function of $W$:
\begin{equation}\label{eq:p}
\displaystyle
\underset{y\in\mathcal{U}(x^*)}\argmin~ W(y;x^*,\bv u^*) = \underset{y\in\mathbb{R}^d}\argmin~\widetilde{W}_\rho(y;x^*,\bv u^*) 
\end{equation}
 where $\rho>0$ is the  penalty factor. The choice of this function is critical, and we will present detailed concepts and theories in the subsequent section.

\section{Iterative Proximal Minimization}\label{IPM}


In this section, we extend the standard IMF by introducing a penalty function to $W$, thereby making equation \eqref{eq:p} applicable and the penalized function $\widetilde{W}_\rho$ both continuous and differentiable. These alterations enable us to formulate a multi-agent game and establish that the Nash equilibrium of this game corresponds to the saddle point of the potential function $V$.

\subsection{Iterative Proximal Minimization}

We propose the following modified IMF with the proximal penalty and call it Iterative Proximal Minimization (``IPM'' in short):
\begin{equation} \label{modified_IMF}
    \left\{\begin{array}{l}
    \bv u^{t}=\argmin_{\|\tilde{\bv u}\|=1} \tilde{\bv u}^{\top} H\left(x^{t}\right) \tilde{\bv u} =\bv v_1(x^t),\\
    x^{t+1}=\argmin_{y\in\mathbb{R}^d}\widetilde{W}_\rho(y;x^t,\bv u^t),
    \end{array}\right.
    \end{equation}
where
    \begin{equation}
    \label{WK}
    \begin{split}
        \widetilde{W}_\rho(y;x, \bv {u}) :&= {W}(y;x, \bv {u}) + \rho\  d(x,y)\\
        &=(1-\alpha)V(y) + \alpha V\left(y-\bv u \bv u^\top(y-x)\right)-\beta V\left(x + \bv u \bv u^\top(y-x)\right) + \rho\  d(x,y)
    \end{split}
    \end{equation}
    with a positive constant $\rho>0$. Here $d(x,y)$ is a penalty function on $\mathbb{R}^d\times \mathbb{R}^d$ satisfying the following assumptions:

\begin{asm}[Penalty Function] \label{asmd}
~
    \begin{enumerate}[label=(\alph*), ref=\ref{asmd}\alph*]
       \item \label{asmd1} For any $x$, $d(x,y)$ is   convex and $C^2$ in $y$, and 
       \[
       \nabla^2_y d(x,y) = \bv 0 ~ \text{ if and only if }~y=x;
       \]
       \item \label{asmd2} $ d(x,y)\geq 0$ for all $x,y$;  and $   d(x,y)=0 $ if and only if $x=y$; 
    \item \label{asmd3}
    For  any constant $\epsilon>0$, there exists a  positive constant  $ \bar{\lambda}_\epsilon>0$, such that  
\[\inf_{\|x-y\| \ge \epsilon}\lambda_{\min}(\nabla^2_y d(x,y))\geq\bar{\lambda}_\epsilon.\]
    \end{enumerate}
   \end{asm}
\begin{rem} \label{rem1}
The first condition says that 
$d(x,y)$ is convex but not  strongly convex in $y$.
So  the quadratic function 
$\|x-y\|^2$ does not satisfy the first condition.
The second condition implies that 
$\nabla_y d(x,y)= \bv 0$ at $y=x$.
The third condition implies 
the strong convexity in $y$ outside any ball centered at $x$.
The  example of quartic $d(x,y) = \|x-y\|_4^4 = \sum_{i=1}^d(x_i-y_i)^4$ satisfies all conditions in  {\bf Assumption} \ref{asmd}. 
\end{rem}

        {\bf Assumption} \ref{asmd1} and {\bf Assumption} \ref{asmd2} above guarantee that the introduction of $d$ to  $\widetilde{W}_\rho$ in \eqref{modified_IMF} will not change the fixed point of the original IMF. 
   {\bf Assumption} \ref{asmd3} is for convexification of    $\widetilde{W}_\rho(y;x,\bv v_1(x))$ for $y\in\mathbb{R}^d$ so that $\min_{y\in\mathbb{R}^d}\widetilde{W}_\rho(y;x,\bv v_1(x))$
is a strictly convex  problem  with a unique solution under {\bf Assumption} \ref{asmd} and the following {\bf Assumption} \ref{asm}. 
 The following assumptions are about  the potential function $V$:

    \begin{asm}[Potential Function] \label{asm}
    ~
        \begin{enumerate}[label=(\alph*), ref=\ref{asm}\alph*]
            \item \label{asm1} $V\in {C}^3(\mathbb{R}^d)$ and is Lipschitz continuous with the Lipschitz constant $\Lip(V)$;
            \item \label{asm2} $V$ has a non-empty and finite set of index-1 saddle points, denoted by $\mathcal{S}_1$;
            \item \label{asm3}  $\nabla^2 V$ is bounded uniformly. That is, there exist two constants $\bar{\lambda}_L,\bar{\lambda}_U$ such that $\bar{\lambda}_L \leq \lambda_{\min}(\nabla^2 V(x))\leq \lambda_{\max}(\nabla^2 V(x)) \leq \bar{\lambda}_U$ for all $x\in \mathbb{R}^d$;
             \item \label{asm4} All stationary points of the potential function $V$ are non-degenerate. 
        \end{enumerate}
    \end{asm}
  
\subsection{Saddle Points and Nash Equilibrium}

In this section, we introduce a differential game, outlining a specific set of agents, action space, and cost functions associated with  the modified auxiliary function, $\widetilde{W}_\rho$. Our goal is to establish the equivalence between the Nash equilibrium of this game and the index-1 saddle point of the potential function $V$.  To accomplish this, we will need to validate certain properties of the modified auxiliary function $\widetilde{W}_\rho$.  We start with 
 a mapping $\Phi_\rho:\mathbb{R}^d\rightarrow 2^{\mathbb{R}^d}$ as the set of global minimizers of $\widetilde{W}_\rho(y;x,\bv v_1(x))$ for a given $x\in\mathbb{R}^d$.
\begin{dfn}\label{phi}
Let $\Phi_\rho:\mathbb{R}^d\rightarrow 2^{\mathbb{R}^d}$ be the (set-valued) mapping defined as
\begin{equation}\label{PHI}
x\mapsto   \Phi_\rho(x):=\argmin_{y\in\mathbb{R}^d}\widetilde{W}_\rho(y;x,\bv v_1(x))  
\end{equation}
If the optimization problem has no optimal solution, $\Phi_\rho(x)$ is defined as an empty set. 
If the optimal solution is not unique, then $\Phi_\rho(x)$ is defined as the collection of all global optimal solutions.
\end{dfn}

We have the subsequent characterization of the fixed points of $\Phi_\rho$, as well as the set of index-1 saddle points pertaining to the potential function $V$.

\begin{thm}\label{lem:fs}  If  {\bf Assumption} \ref{asmd} and  \ref{asm} hold, then we have that
 \begin{enumerate}
    \item[(1)]  
    There exists a positive constant $\bar{\rho}$ depending on the bounds of  $\nabla^2V$, $\alpha$ and $\beta$ only,
    such that  
    \[
    \Phi_\rho(x^*)= \{x^*\}, \quad \forall x^*\in \mathcal{S}_1, ~
\rho\geq \bar{\rho}.
    \]
    \item[(2)] 
    If $\Phi_\rho(x^*)=\{x^*\}$ with  $\rho>0$, then $x^*\in \mathcal{S}_1$.
\end{enumerate}
\end{thm}
We remark that the selection of the penalty factor $\bar{\rho}$ is independent of the saddle point $x^*$.  This implies that, computationally, there is no prerequisite for the algorithm to possess {\it a prior} information of the saddle point.

The proof of Theorem \ref{lem:fs} requires the subsequent proposition and theorem.
\begin{lem}\label{pro:lip}
Let $V$ be a continuous Lipschitz function with the Lipschitz constant $\Lip(V)$, then for any point $x\in \mathbb{R}^d$, any unit vector $\bu \in \mathbb{S}^{d-1}$, and  two constants $\alpha$, $\beta$, we have that the function $W$ defined in \eqref{W} is also  Lipschitz uniformly for $x$ and $\bv v$,  with the Lipschitz constant
$\operatorname{Lip}(W) = (\abs{\alpha}+\abs{1-\alpha}+\abs{\beta} )\Lip(V)$.
\end{lem}

 \begin{proof}
 For each $x\in\mathbb{R}^d$ and   $y_1,y_2\in\mathbb{R}^d$, we have
\begin{align*}
\begin{split}
     & \| W  (y_1;x,\bu)  -W(y_2;x,\bu)\|  \\
  \leq  & \abs{1-\alpha} \|V(y_1) - V(y_2)\|  + \abs{\alpha}
  \norm{
  V\left(y_1-\bu \bu^\top(y_1-x)\right)-V\left(y_2-\bu \bu^\top(y_2-x)\right)
  }
    \\
    & +\abs{\beta}\|V(x + \bu\bu^\top(y_1-x)) - V(x + \bu\bu^\top(y_2-x))\|
    \\
     \leq &\abs{1-\alpha}\Lip(V)\|y_1-y_2\|  + \abs{\alpha}\Lip(V)\|  ( \mathbf{I}-\bu\bu^\top) (y_1-y_2)\|+\abs{\beta}\Lip(V)\|\bu\bu ^\top(y_1-y_2)\|.
\end{split}
\end{align*}
Then the conclusion is clear since 
$\norm{\bv P  x} \le \norm{x} $ for any projection matrix $\bf P$.
\end{proof}

The following theorem elucidates the existence and uniqueness of the minimizer in equation \eqref{PHI} for $x$ within the index-1 region $\Omega_1$ . It further asserts that this unique minimizer continues to reside inside $\Omega_1$.

\begin{thm}\label{pro:convex}
Suppose {\bf Assumption} \ref{asmd} and {\bf Assumption} \ref{asm} hold, then
for any compact subset ${\Omega}^\prime_1$ of the index-1 region $\Omega_{1}$ and two constants $\alpha+\beta>1$, there exists a positive constant $\bar{\rho}>0$ depending on $\alpha, \beta$, and $\Omega^\prime_1$, such that for all $\rho>\bar{\rho}$, the following optimization problem of $y$,
    \begin{equation}\label{eq:min}
         \begin{split}
\min_{y\in\mathbb{R}^d}\widetilde{W}_\rho(y;\, x,\bv v_1(x)) = {W}(y;\,x,\bv v_1(x)) + \rho\cdot d(x,y)\\
    \end{split}
    \end{equation} 
has a unique solution $\hat{x}\in\Omega_1$ for each $x \in {\Omega}^\prime_1$, i.e., $\Phi_\rho(x) = \{\hat{x}\}\subset\Omega_1$.
\end{thm}

\begin{proof}
To prove our conclusion,  we will claim the optimization \eqref{eq:min} is a strictly convex problem by showing that at a sufficiently large $\rho$,  $\widetilde{W}_\rho$ is a strictly convex function of $y$ in $\mathbb{R}^d$ uniformly for  $x \in \Omega_1^\prime$.
    This can be proved by studying   the minimal eigenvalues of the Hessian matrix. 
    The Hessian matrix of $\widetilde{W}_\rho$   with respect to $y$ is 
    \begin{equation} \label{eq:HWT}
    \begin{split}
    \widetilde{H}_\rho(y;x)&:=\nabla_y^2  \widetilde{W}_\rho(y; x,\bv v_1(x))
    \\
    &=(1-\alpha) H(y) +\alpha\big[\mathbf{I} - \Pi_1(x)\big]H(y-\Pi_1(x)(y-x))\big[\mathbf{I} - \Pi_1(x)\big]\\
    &\qquad- \beta\Pi_1(x) \, 
    H\left(x + \Pi_1(x)(y-x)\right)
    \Pi_1(x)
    + \rho\nabla^2_y d(x,y)\\
    & =:\mathcal{H}(y;x) + \rho\nabla^2_y d(x,y),
    \end{split}
    \end{equation}    
    where  $\Pi_1(x) := \bv v_1(x)\bv v_1(x)^\top$, and the new symmetric matrix
     \begin{equation} \label{Hyx}
    \begin{split}
        \mathcal{H}(y;x) &:= (1-\alpha)H(y) +\alpha\big[\mathbf{I} - \Pi_1(x)\big]H(y-\Pi_1(x)(y-x))\big[\mathbf{I} - \Pi_1(x)\big]\\
    &\qquad- \beta\Pi_1(x) H(x + \Pi_1(x)(y-x))\Pi_1(x).
    \end{split}
    \end{equation}
    By the inequality 
    
    \begin{equation}
        \label{614}
     \lambda_{\min}\left(\widetilde{H}_\rho(y;x)\right)\geq \lambda_{\min}(\mathcal{H}(y;x)) + \rho\cdot \lambda_{\min}(\nabla^2_y d(x,y)),
    \end{equation} we focus on $\lambda_{\min}(\mathcal{H}(y;x))$ first.
    
    Given a compact set $\Omega_1^\prime$ in the index-1 region $\Omega_1$, 
    we first fix a point $x \in \Omega_1^\prime\subset \Omega_1$.
Then $\lambda_1(x)<0<\lambda_2(x)$.
Note the eigenvectors of 
\begin{align*}
    \mathcal{H}(x;x) = (1-\alpha)H(x) + \alpha(\mathbf{I}-\Pi_1(x))H(x)(\mathbf{I}-\Pi_1(x)) - \beta\Pi_1(x)H(x)\Pi_1(x)
\end{align*}
coincide with the eigenvectors of the Hessian matrix $H(x) = \nabla^2 V(x)$, because  for $i\neq1$, we have
\begin{align*}
\begin{split}
    \mathcal{H}(x;x)\bv v_i(x) &= (1-\alpha)H(x)\bv v_i(x) +\alpha(\mathbf{I}-\Pi_1(x))\lambda_i(x)\bv v_i(x) + 0\\
    &=(1-\alpha)\lambda_i(x)\bv v_i(x) + \alpha\lambda_i(x)\bv v_i(x)\\
    &=\lambda_i(x)\bv v_i(x)
\end{split}
\end{align*}
and at $i=1$, 
\begin{align*}
\begin{split}
    \mathcal{H}(x;x)\bv v_1(x) & =(1-\alpha)\lambda_1(x)\bv v_1(x) + 0 - \beta\lambda_1(x)\bv v_1(x)\\
    &=(1-\alpha-\beta)\lambda_1(x)\bv v_1(x).
\end{split}
\end{align*}
Therefore,  the eigenvalues of $\mathcal{H}(x;x)$ are given by
\[
\{(1-\alpha-\beta)\lambda_1(x),\lambda_2(x),\cdots,\lambda_d(x)\}
\]
and they are all strictly positive since
$\alpha+\beta>1$ and $x\in \Omega_1$.
Therefore $\mathcal{H}(x;x)$ is positive definite. In addition, since $V\in{C}^3(\mathbb{R}^d)$,
 for each $x$ we have that $\lambda_{\min} (\mathcal{H}(y;x))>0$ for all $y$ inside a ball neighbourhood $\mathcal{B}_{\varepsilon_x}(x)$  with a radius $\epsilon_x>0$ depending on $x$, due to the continuity of $\nabla^2 V$.
Specifically,  
$\epsilon_x :=\frac12 \sup\{\epsilon>0:~ \min_{\|y-x\|\leq\epsilon}\lambda_{\min} (\mathcal{H}(y;x))>0\}$. 
  Pick up the smallest radius  for all $x$ in $\Omega_1^\prime$:
$\bar{\varepsilon}=\min_{x\in \Omega_1^\prime} \epsilon_x>0,$
  which is strictly positive since $\Omega_1^\prime$ is compact.
  This means that   
 \[
  \inf_{x\in \Omega'_1}\inf_{\|y-x\|\le \bar{\epsilon}} \lambda_{\min}( \mathcal{H}(y;x) ) >0,
    \]
   and then due to \eqref{614},  for any $\rho\ge 0$, the Hessian matrix $\widetilde{H}_\rho$
    satisfies the same condition 
    \begin{equation}
    \label{eq:620}
      \inf_{x\in \Omega'_1}\inf_{\|y-x\|\le \bar{\epsilon}} \lambda_{\min}( \wt{H}_\rho(y;x) ) >0,
    \end{equation}
      since $\nabla^2_y d(x,y)\succeq 0$ due to {\bf Assumption} \ref{asmd1}.
    
    In order to show $\widetilde{H}_\rho(y;x)\succ 0$ for all $y\in \mathbb{R}^d$, we shall choose a sufficiently large penalty factor $\rho$.
   By {\bf Assumption} \ref{asmd3},  there exists a constant $\bar{\lambda}_{\bar{\varepsilon}}>0$, such that $\lambda_{\min}(\nabla^2_y d(x,y)) \ge  \bar{\lambda}_{\bar{\varepsilon}}$  for any $x,y$ satisfying $ \|y-x\|\ge \bar{\varepsilon}$. 
Recall that from {\bf Assumption} \ref{asm3}, the Hessian matrix of potential function $V$ is bounded everywhere. Let $\bar{\lambda} = \max\{|\bar{\lambda}_L|,|\bar{\lambda}_U|\}$, then by noting \eqref{Hyx},  we have the lower bound of the minimal eigenvalue  for all $\|x-y\| \ge \bar{\varepsilon}$
\begin{align*}\label{422}
\begin{split}
\lambda_{\min}\left(\wt{H}_\rho(y;x)\right)
&\geq \lambda_{\min}(\mathcal{H}(y;x)) + \rho\lambda_{\min}(\nabla^2_y d(x,y))\\
&\geq -(|1-\alpha| + |\alpha| + |\beta|) \bar{\lambda} + \rho\bar{\lambda}_{\bar{\epsilon}}.
\end{split}
\end{align*}
Let  $\rho>\bar{\rho} := {(1+2|\alpha|+|\beta|)\bar{\lambda}}/{\bar{\lambda}_{\bar{\varepsilon}}}>0$, 
then
\[ 
\displaystyle  \inf_{x\in \Omega'_1}\inf_{\|y-x\|>\bar{\epsilon}}
~\lambda_{\min}\left(\wt{H}_\rho(y;x)\right)>0.
\]
Therefore, together with \eqref{eq:620} we conclude that when $\rho>\bar{\rho}$, 
\begin{equation}
\label{eq:621}
 \inf_{x\in \Omega'_1} \inf_{y\in  \mathbb{R}}~\lambda_{\min}\left(\wt{H}_\rho(y;x)\right)>0.   
\end{equation}
That is, $\widetilde{W}_\rho(y;x,\bv v_1(x))$ is strongly  convex in $y$   and the minimization problem in equation 
\eqref{eq:min}  has a unique solution $\Phi_\rho(x)$ for all $x\in \Omega^\prime_1$. 

Lastly, we shall show that for a  sufficiently large $\rho$, the unique solution 
\begin{equation*}
         \hat{x}
          = \argmin_{y\in\mathbb{R}^d} \left( W(y; x,\bv v_1(x)) + \rho\cdot d(x,y)\right)
\end{equation*}
is in the index-1 region $\Omega_1$ for any $x\in\Omega_1^\prime$. 
It suffices to show that if $\rho$ is large enough, then for any $x\in \Omega_1^\prime$ and $y^\prime \notin   \Omega_1$, there exists $y\in \Omega_1$ such that 
\begin{equation}\label{ine1}
      W(y; x,\bv v_1(x)) + \rho\cdot d(x,y) < W(y^\prime; x, \bv v_1(x)) + \rho\cdot d(x,y^\prime).
\end{equation}
The idea is to search for such a point $y$ along the line segmentation between $x$ and $y^\prime$.

Let $\delta:=\operatorname{dist}(\Omega_1', \partial \Omega_1)$ be the distance between  $\Omega_1^\prime$ and  the boundary of $\Omega_1$, 
and assume that $\delta =\operatorname{dist}(\Omega_1^\prime,  \Real^d \setminus \Omega_1)$.
Then $\delta>0$   because $\Omega_1$ is an open set and $\Omega_1^\prime$ is a compact subset of $\Omega_1$.
We define the $\delta/2$-neighborhood   of $\Omega'_1$ and its boundary:
$$A_\delta:=\{z: \operatorname{dist}(z, \Omega'_1) \le \delta /2 \}\subset \Omega_1
\text{ and }\partial A_\delta =\{z: \operatorname{dist}(z, \Omega'_1) = \delta /2 \}.
$$
It is easy to show that
$\operatorname{dist}(A_\delta, \partial \Omega_1)\ge \delta/2$: 
There exist two points $a^*_\delta \in \partial A_\delta$ and $b\in\partial\Omega_1$ such that 
$\operatorname{dist}(A_\delta, \partial \Omega_1) =\operatorname{dist}(\partial A_\delta, \partial \Omega_1)=\|a_\delta^*-b\|$,  and one point $a^*\in\Omega_1^\prime $ such that   $\|a^*-a_\delta^*\|=\delta/2$
since $\operatorname{dist}(a_\delta^*, \partial \Omega_1) =\delta/2$, then
$$\operatorname{dist}(A_\delta, \partial \Omega_1) =\|a_\delta^*-b\| \ge \|a^*-b\|-\|a^*-a_\delta^*\| = \|a^*-b\|-\delta/2 \ge \delta/2.$$

Therefore, for any $y^\prime\notin\Omega_1$ and $x\in\Omega_1^\prime$, we 
define $y$ as an intersection point on the boundary $\partial A_\delta$ and the line   segment between $y^\prime$ and $x$ such that the line segment between 
$y$ and $y'$ are all outside of the $A_{\delta}$.  Since $y\in \partial A_\delta$, $\|y-x\|\ge \operatorname{dis}(y, \Omega_1^\prime) =\delta /2  $ and $\|y-y^\prime\|\ge \operatorname{dist}(A_\delta, \partial\Omega_1)\ge \delta/2$.
Since the line segment $y(t)=y+t(y^\prime-y)$, $0\le t\le 1$ in outside of $A_\delta$, we have $\min_{t\in[0,1]}\|y(t)-x\|\ge \delta/2$, which,  by 
   {\bf Assumption} \ref{asmd3}, gives    that 
$$\min_{t\in[0,1]}\lambda_{\min}(\nabla^2_y d(y(t),x) \ge \bar{\lambda}_{\frac{\delta}{2}}>0$$

Therefore,
\begin{equation} \label{eq311}
\begin{split}
    d(y^\prime,x) - d(y,x) 
    & = \nabla_y d(y,x)^\top(y^\prime-y)
    +\int_0^1 \int_0^t (y^\prime-y) ^\top \nabla^2_y d(y(s),x)  (y^\prime-y) \dd s \dd  t
    \\
    &\geq    \nabla_y d(y,x)^\top(y^\prime-y)+
    \frac{1}{2}\bar{\lambda}_{\frac{\delta}{2}}\|y^\prime - y\|^2
    \\
    &\geq
     \nabla_y d(y,x)^\top(y^\prime-y)+ \frac{1}{4}\delta\bar{\lambda}_{\frac{\delta}{2}} \|y^\prime-y\| 
\end{split}
\end{equation}

    In addition, by  {\bf Assumption} \ref{asmd1} and \ref{asmd2} about the convexity, 
$0 = d(x,x) \geq d(y,x) + \nabla_y d(y,x)^\top(x-y),
$ which leads to $
\nabla_y d(y,x)^\top(y - x)\geq  d(y,x)\geq 0,$ and 
\begin{equation}
    \nabla_y d(y,x)^\top(y^\prime - y)\geq 0,
\end{equation}
as $y'-y$ and $y-x$ are in the same direction. Therefore, by \eqref{eq311} we have shown that 
\[
d(y^\prime,x) - d(y,x) \geq D_\delta \|y^\prime - y\|
\]
where    $D_\delta := \frac{1}{4} \delta\bar{\lambda}_{\frac{\delta}{2}} $.
Together with Lemma \ref{pro:lip},  we have 
\begin{equation}
    \begin{split}
        \widetilde{W}_\rho  (y^\prime; x, \bv v_1(x)) &= W(y^\prime; x,\bv v_1(x)) + \rho\cdot d(y^\prime,x)\\
        &\geq W(y;x, \bv v_1(x)) - \Lip(W)\|y^\prime - y\| + \rho\cdot\left(d(y^\prime,x) - d(y,x) + d(y,x)\right)\\
        &\geq W(y;x, \bv v_1(x))+\rho\cdot d(y,x) + (\rho D_\delta - \Lip(W))\|y^\prime - y\|\\
        &=\widetilde{W}_\rho(y; x, \bv v_1(x)) + (\rho D_\delta - \Lip(W))\|y^\prime - y\|.
    \end{split}
\end{equation}

For any $\rho>\frac{\Lip(W)}{D_\delta}$, we have
\begin{equation}
    \widetilde{W}_\rho(y^\prime; x, \bv v_1(x)) >\widetilde{W}_\rho(y; x, \bv v_1(x))\geq \min_{y\in\Omega_1}\widetilde{W}_\rho(y; x, \bv v_1(x)).
\end{equation}
Then together with equation \eqref{eq:621}, we have proved that the uniqueness of the global minimizer $$\hat{x} = \underset{y\in\mathbb{R}^d}{\argmin}~\widetilde{W}_\rho(y;x,\bv v_1(x))$$ and $\hat{x}\in\Omega_1$ is within the index-1 region for any compact $x\in\Omega^\prime_1 \subset \Omega_1$ for sufficiently large 
\begin{align*}
\rho >\bar{\rho} := \max(\frac{\Lip(W)}{D_\delta},\frac{(1+2|\alpha|+|\beta|)\bar{\lambda}}{\bar{\lambda}_{\bar{\varepsilon}}}).
\end{align*}
\end{proof}

\begin{rem} $\bar{\rho}$ depends on the uniform bound  of Hessian $\nabla^2 V$, two constants  $\alpha$, $\beta$, and the compact  subset  $\Omega'_1\subseteq \Omega_1  $.
\end{rem}

With   Theorem \ref{pro:convex}, we are ready to present the proof of Theorem \ref{lem:fs} to  establish the equivalence between the fixed points of the map $\Phi_\rho(\cdot)$ and the index-1 saddle points of the potential function $V(\cdot)$.


\begin{proof}[{\bf Proof of Theorem \ref{lem:fs}}]
``Proof of Statement (1)'': \\
From Theorem \ref{pro:convex}, we know that for each index-1 saddle point $x_i^*\in\mathcal{S}_1$, there exists a positive $\bar{\rho}_{i}$ such that $\widetilde{W}_\rho(y;x_i^*,\bv v_1(x_i^*))$ is a strictly convex function of $y\in\mathbb{R}^d$ for all $\rho>\bar{\rho}_i$. Together with {\bf Assumption} \ref{asm2} and by setting  $\bar{\rho}:=\max_i \bar{\rho_i}>0$,  we have that 
for any ${x}^*\in\mathcal{S}_1$, then for any $\rho>\bar{\rho}$, $\widetilde{W}_\rho(y;x^*,\bv v_1(x^*))$ is a convex function of $y\in\mathbb{R}^d$. 
 We only need to show that the first order condition holds.
Note that
\begin{align}
    \label{1st_eq}
    \nabla_y \widetilde{W}_\rho(y;x, \bv v_1(x)) = &(1-\alpha)\nabla V(y)  + \alpha(\mathbf{I}- \Pi_1)\nabla V(y- \Pi_1(y-x)) \nonumber\\
    & - \beta  \Pi_1 \nabla V(x+ \Pi_1(y-x)) + \rho\nabla_y d(x,y),
\end{align}
where $\Pi_1 =\Pi_1(x)= \bv v_1(x)\bv v_1(x)^\top$,
we get
$\nabla_y \widetilde{W}_\rho(y;x, \bv v_1(x))\vert_{y=x} = \Big[\mathbf{I}-(\alpha+\beta)\Pi_1(x)\Big]\nabla V(x)$ by {\bf Assumption} \ref{asmd2}.
So, $ \nabla_y \widetilde{W}_\rho(y;x^*, \bv v_1(x^*))\vert_{y=x^*} = \mathbf{0}
$  thanks to $\nabla V(x^*)= \bv 0$. 

``Proof of Statement (2) '':\\
Now given that $\Phi_\rho(x^*) = \{x^*\}$, i.e., $x^*$ is  the unique minimizer in \eqref{PHI}, we want to show $x^*$ is an index-1 saddle point. This is achieved  by showing that the first order condition $ \nabla_y \widetilde{W}_\rho(y;x^*, \bv v_1(x^*))\vert_{y=x^*} = \mathbf{0}$ holds and the Hessian matrix $\nabla^2_y\widetilde{W}_\rho(y;x^*,\bv v_1(x^*))|_{y = x^*}$ is positive semi-definite. 
By the first order condition, we have
\begin{equation}
\label{f_order}
\begin{split}
&\nabla \widetilde{W}_\rho(y;x^*,\bv v_1(x^*))|_{y = x^*} =\Big[\mathbf{I}-(\alpha+\beta)\Pi_1(x^*)\Big]\nabla V(x^*) = \mathbf{0}.
\end{split}
    \end{equation}
     
    Since $\alpha+\beta>1$, $\nabla V(x^*)= \bv 0$ holds.
    For the second order condition, by \eqref{eq:HWT}, we have
    \[ \nabla^2 \widetilde{W}_\rho(y;x^*,\bv v_1(x^*))|_{y = x^*}
    = \mathcal{H}(x^*;x^*) \succeq \bv 0.
    \]
    From the proof of theorem \ref{pro:convex}, we know that the eigenvalues of the Hessian matrix
    \[
    \mathcal{H}(x^*;x^*) = \nabla^2 \widetilde{W}_\rho(y;x^*,\bv v_1(x^*))|_{y = x^*}
    \]
    is given by
    \[
    \{(1-\alpha-\beta)\lambda_1(x^*),\lambda_2(x^*),\cdots,\lambda_d(x^*)\}.
    \]
    Then we have that
    \[
    \nabla^2 \widetilde{W}_\rho(y;x^*,\bv v_1(x^*))|_{y = x^*}\succeq \bv 0
    \]
    is equivalent to
    $
    \lambda_1(x^*) \le 0\le \lambda_2(x^*)\leq\cdots\leq\lambda_d(x^*).
    $
    By {\bf Assumption} \ref{asm4}, we know that $x^*$ is non-degenerate with non-zero eigenvalues,  thus  
    \[
    \lambda_1(x^*)<0<\lambda_2(x^*)\leq\cdots\leq\lambda_d(x^*).
    \]
    Together with $\nabla V(x^*) = \mathbf{0}$, we conclude that $x^*$ is an index-1 saddle point of the potential function $V(\cdot)$.
The proof of theorem \ref{lem:fs} is now completed.
\end{proof}

\subsection{Interpretation of IPM as Differential Game  Model}
We follow Section \ref{Game_theory} to further   present  our interpretation of the iterative proximal minimization. 
 For index-1 saddle point, the differential game ${G}_3$ is played by three agents indexed by $\{-1,0,1\}$ with their costs depicted  in Table \ref{G3}. 
Demonstrating that index-1 points indeed correspond to the Nash equilibrium of a differential game, we can establish  the equivalence between fixed points of $\Phi_\rho$ and the Nash equilibrium of the differential game $G_3$.
\begin{table}

    \begin{center}
	\begin{tabular}{|c|c|c|}
		\hline
		agent&Action Variable  &Cost function\\
		\hline
	    ``-1'' & $y\in\mathbb{R}^d$ & $\widetilde{W}_\rho(y;x,\bv u)$\\
	    ``0'' & $x\in\mathbb{R}^d$ & $\frac{1}{2}\|x-y\|^2$\\
	     ``1'' & $\bv u\in\mathbb{S}^{d-1}$ & $\bv u^\top H(x)\bv u$\\
		\hline	
	\end{tabular}
    \end{center}
    \caption{The definition of the 3-agent game $G_3$,where $\mathbb{S}^{d-1}$ is the unit $L_2$ sphere in $\mathbb{R}^d$.}
    \label{G3}
     \end{table}

\begin{thm}\label{lem:ns}
For any $\rho>\bar{\rho}$, an action profile $(y^*,x^*,\bv u^*)$ is a Nash equilibrium of $G_3$ defined in Table \ref{G3} if and only if  $y^*=x^*, \bv u^*=\bv v_1(x^*)$ and $x^*$ satisfies  $\Phi_\rho(x^*) = \{x^*\}$, where $\bar{\rho}$ is determined in Theorem \ref{pro:convex}.
 \end{thm}
\begin{proof}
    The proof mainly    follows the definition of the Nash equilibrium:
    \begin{enumerate}
        \item ``$\Rightarrow$'':
        If an action profile $(y^*,x^*,\bv u^*)$ is a Nash equilibrium of $G_3$, by the definition of Nash equilibrium, we have
        \[
        x^* = \argmin_{x\in \mathbb{R}^d} \frac{1}{2}\|x - y^*\|^2 = y^*
        \]
        and 
        \[
        \bv u^* = \argmin_{\bv u \in \mathbb{S}^{d-1}} (\bv u^*)^\top H(x^*) \bv u^*
        =\bv v_1(x^*). \]
        Therefore
        \[
        x^* = \argmin_{y\in\mathbb{R}^d}\widetilde{W}_\rho(y; x^*, \bv v_1(x^*)),
        \]
        which is equivalent to $\Phi_\rho(x^*) = \{x^*\}$ as a result of the fact that $\widetilde{W}_\rho(\cdot; x^*, \bv x^*)$ is strictly convex for any $\rho>\bar{\rho}$.
        \item ``$\Leftarrow$'':
        If $y^*=x^*, \bv u^*=\bv v_1(x^*)$ and $x^*$ satisfies  $\Phi_\rho(x^*) = \{x^*\}$, then we have 
        \[
        y^* = x^* = \argmin_{y\in\mathbb{R}^d}\widetilde{W}_\rho(x,x^*,\bv v_1(x^*)),\quad x^* = y^* = \argmin_{x\in\mathbb{R}^d}\frac{1}{2}\|x-y^*\|
        \]
        and
        \[
        \bv u^* = \bv v_1(x^*) = \argmin_{\bv u\in\mathbb{S}^{d-1}}\bv u^\top H(x^*)\bv u
        \]
        which is the Nash equilibrium of $G_3$ by definition.
    \end{enumerate}
\end{proof}
  


 Theorem \ref{lem:fs} and Theorem \ref{lem:ns} together directly lead to
the main result of  Theorem \ref{thm:main}.
    \begin{thm}\label{thm:main}
Suppose that {\bf Assumption} \ref{asmd} and {\bf Assumption} \ref{asm}  hold. There exists a positive constant $\bar{\rho}>0$, such that for any sufficiently large penalty factor $\rho>\bar{\rho}$, we have the following statements hold:
    \begin{enumerate}
        \item If $x^*$ is an index-1 saddle point of the potential function $V$ satisfying equation \eqref{eq:lam}, then $(x^*, x^*, \bv v_1(x^*))$ is a strict Nash equilibrium of $G_3$;
        \item If $(x,y,\bv u)$ is a Nash equilibrium of $G_3$, then  $y=x,\ \bv u(x) = \bv v_1(x)$, and  $x$ is a non-degenerate index-1 saddle point of the potential function $V$. 
        \end{enumerate}

    \end{thm}

 \subsection{Algorithm and Discussion of Convergence}
 

The main feature of our proposed approach is the incorporation   of a non-quadratic penalty function $d$, which satisfies the criteria of {\bf Assumption} \ref{asmd}, into the original auxiliary function  in  the IMF. 
This proximal minimization method manifests a well-founded game theory model and highlights the numerical superiority of increasing the robustness of existing algorithms based on the IMF. This advantage is supported by the forthcoming numerical examples. 

Our modification of the existing algorithm is straightforward to implement and, for completeness, the main steps are described in Algorithm \ref{alg:algorithm1} with a constant $M$ steps to solve the sub-problem. But it should be emphasized that, in practical applications, inner minimization subproblems can only be  solved  approximately,  with a suitable stopping criterion \cite{IMA2015}.


\begin{algorithm}
\caption{Iterative Proximal Minimization Algorithm}
\label{alg:algorithm1}
\begin{algorithmic}
\STATE{{\bf{Input}}: initial guess $x^0$, $\rho_0>0$, $tol>0$.}
\STATE{ {\bf{Output}}: saddle point $x$.}
\STATE{{\bf{begin}}}
\STATE{Solve the min-mode $ \displaystyle \bv u^0=\argmin_{\|\bv u\|=1} \bv u^\top H(x^0) \bv u$;}
\STATE{$t=0$;}
\WHILE{$ |\nabla V(x^t)| > tol$}
\STATE{$y^{0,t}\leftarrow x^t$; }
\FOR{$i=0, 1, 2, \cdots, M-1$}
    \STATE{$y^{i+1,t} = y^{i,t} - \Delta t*\nabla \widetilde{W}_{\rho_t}(y^{i,t};x^t, \bv u^t)$; \qquad\  //   gradient descent for  $\widetilde{W}_{\rho_t}$ in \eqref{WK}}
\ENDFOR
\STATE{$t=t+1$;}
\STATE{	$x^t \leftarrow y^{M,t}$;}
\STATE{	$\bv u^t  = \argmin_{\|\bv u\|=1}{\bv u^\top H(x^t) \bv u}$; \qquad\qquad\qquad  // solve the minimal eigenvector }
\ENDWHILE
\RETURN $x^t$
\end{algorithmic}
\end{algorithm}

We discuss the convergence aspect of the algorithm for a finite $M$.
In practical terms, the solution $y^{M,t}$ merely provides an approximation to the true minimizer $\Phi_{\rho_t}(x^t)$ at the $t$-th iteration for the subproblem.
If certain Lipschitz conditions for the mapping $\Phi_{\rho_t}(\cdot)$ hold, we can demonstrate the convergence of the sequence ${x^t}$ that the algorithm generates for any initial assumption $x^0\in\Omega_1$. We note that in realistic applications, such a Lipschitz condition for
$\Phi_{\rho_t}$ may not be easy to verify.

\begin{thm}
For any initial guess $x^0$ inside the index-1 region $\Omega_1$, 
let $\{x^t\}$ generated from Algorithm \ref{alg:algorithm1} with a sequence of the penalty factors 
$\{\rho_t\}$ satisfying (as  in Theorem \ref{lem:fs}) $\rho_t>\bar{\rho} := \max(\frac{\Lip(W)}{D_\delta},\frac{(1+2|\alpha|+|\beta|)\bar{\lambda}}{\bar{\lambda}_{\bar{\varepsilon}}})$
and  that $\Phi_{\rho_t}$ is $L_t$-Lipschitz continuous: \[
\|\Phi_{\rho_t}(x)-\Phi_{\rho_t}(y)\|\leq L_t\|x-y\|,\quad \forall x,y.
\]
Assume that the inexact minimizer $y^{M,t}$ in Algorithm \ref{alg:algorithm1} satisfies the numerical tolerance $\epsilon$: 
\[\displaystyle
\|y^{M,t} - \operatorname{argmin}_{y\in\mathbb{R}^d} \widetilde{W}_{\rho_t}(y;x^t,\bv v_1(x^t))\|\leq \epsilon,
\]
then for after $N$ iterations,  
\[
\|x^N-x^*\|\leq \|x^0-x^*\|\cdot\prod_{n=1}^N L_n + \epsilon\cdot\sum_{n=1}^N\prod_{j=N-n}^N L_j,
\]
where $x^*$ is an index-1 saddle point of the potential function $V$.
If  $L_t\le L <1$ is satisfied for all $t \in \{1,\cdots,N\}$, then 
$$ \|x^N-x^*\|\leq  L^N \|x^0-x^*\| + \frac{L^N}{1-L}\epsilon.$$
    \end{thm}

If we denote the inexact iteration $x^t \to x^{t+1}$ in Algorithm \ref{alg:algorithm1} with the tolerance $\epsilon$
by an abstract mapping $x\mapsto \hat{g}_{t+1}(x)$, which is regarded as a perturbation of the exact solver mapping $g_t=\Phi_{\rho_t}$, the result in the above theorem, in fact, is a conclusion of the following proposition about the perturbation of fixed-point iteration. 
\begin{pro}
    Let $N>0$ be an integer and  $\{g_n:\mathbb{R}^d\rightarrow\mathbb{R}^d\}_{n=1}^N$ be a sequence of functions such that 
    \begin{enumerate}
        \item there exists a unique $x^* \in \mathbb{R}^d$ such that $g_n(x^*) = x^*$ for all $n=1,2,\cdots,N$;
         
        \item each $g_n(\cdot)$ is Lipschitz: $\|g_n(x) - g_n(y)\|\leq L_n \|x-y\|$ for all $x,y$ and $n=1,2,\cdots,N$.
    \end{enumerate}
    Then for 
    perturbed functions 
    $\hat{g}_n$ satisfying $\|g_n - \hat{g}_n\|_\infty\leq\epsilon_n$ for a sequence of positive number $\{\epsilon_n \} $, 
    we have that for arbitrary $x^0\in\mathbb{R}^d$, the sequence $\{x^n\}_{n=0}^N$ generated by the perturbed iteration
    $
    x^{n+1} = \hat{g}_{n+1}(x^n),
    $
    satisfies 
    \[
    \|x^N - x^*\| \leq \|x^0-x^*\|  \prod_{n=1}^N L_n+ \sum_{n=1}^N
    \qty(\epsilon_n\prod_{j=n+1}^N L_j) + \epsilon_N.
    \]
\end{pro}
\begin{proof}
    This proposition extends the result in  \cite{alfeld1982fixed}, and is proved by mathematical induction:
    \begin{equation*}
        \begin{split}
            \|x^N-x^*\| &= \|\hat{g}_N(x^{N-1}) - {g}_N(x^*)\|=\|\hat{g}_N(x^{N-1}) - g_N(x^{N-1})+g_N(x^{N-1})- g_N(x^*)\|\\
            &\leq \epsilon_N + L_N\|x^{N-1}-x^*\|
            = \epsilon_N + L_N\|\hat{g}_{N-1}(x^{N-2}) - g_{N-1}(x^*)\|\\
            &\leq \epsilon_N + L_N\epsilon_{N-1} + L_N L_{N-1}\|x^{N-2}-x^*\|.
        \end{split}
    \end{equation*}
    Repeat this procedure, and the desired result follows.
\end{proof}

 \subsection{Generalization to high index saddle point}
The conclusions and algorithms can be  easily extended to  non-degenerate index-$k$ saddle points for $k\in\{1,\cdots,d-1\}$ defined below:
\begin{dfn}[Index-$k$ saddle points]
    If $x_k^*$ is  a non-degenerate index-$k$ saddle point of the potential function $V$, then we have
    \begin{enumerate}
        \item $\nabla V(x_k^*) = 0$;
        \item $\lambda_1(x^*_k)\le \lambda_2(x^*_k)\le \cdots\le \lambda_k(x^*_k)<0< \lambda_{k+1}(x^*_k)\leq \cdots\leq\lambda_d(x^*_k)$.
    \end{enumerate}
\end{dfn}

We consider the penalized proximal cost function in the following form:
\begin{equation}\label{p_w}
    \widetilde{W}_\rho(y;x,\bv v_{1:k}(x)) := (1-\alpha)V(x) + \alpha V(y - \bv v_{1:k}\bv v_{1:k}^\top) - \beta V(x + \bv v_{1:k}\bv v_{1:k}^\top(y-x))
    + \rho d(x,y),
\end{equation}
where $\bv v_{1:k} = [\bv v_1,\cdots, \bv v_k]$ is the matrix with columns formed by the $k$ dimensional eigen-space
associated with the  lowest $k$ eigenvalues,
and these eigen-vectors    satisfy the unit norm   $\norm{\bv v_i}=1$ and 
the orthogonal condition $\bv v_i\cdot \bv v_j = \delta_{ij}$.  The replacement of the single vector $\bv v_1$ by this $k$-by-$d$ matrix $
\bv v_{1:k}$ in the only modification for the index-$k$ case.

We also define  the index-$k$ region 
\[
\Omega_k := \{x\in\mathbb{R}^d : \lambda_1(x)\le \lambda_2(x)\le \cdots\le \lambda_k(x)<0< \lambda_{k+1}(x)\leq \cdots\leq\lambda_d(x)
\]
and  let the mapping $\Phi_\rho$ defined in definition \ref{phi} to the case of index-$k$ saddle points
\[
\Phi_\rho(x) = \argmin_{y\in\mathbb{R}^d}\widetilde{W}_\rho(y;x,\bv v_{1:k}(x)). 
\]
Then, we can have the similar conclusion from Theorem \ref{lem:fs} that 
\[
\Phi_\rho(x) = \{x\}\text{ if and only if } x = x^*_k.
\]
The corresponding game $G_{k+2}$ is modified from $G_3$ to ensure the orthogonality of the lowest $k$ eigenvectors.
\begin{center}
	\begin{tabular}{|c|c|c|}
		\hline
		Player&Action&Cost function\\
		\hline
	    ``-1'' & $y\in\mathbb{R}^d$ & $\widetilde{W}_\rho(y;x,\bv u_{1:k})$\\
	    ``0'' & $x\in\mathbb{R}^d$ & $\frac{1}{2}\|x-y\|^2$\\
	     ``1'' & $\bv u_1\in\mathbb{S}^{d-1}$ & $\bv u_1^\top H_1(x)\bv u_1  $\\
	     ``2'' & $\bv u_2\in\mathbb{S}^{d-1}$ & $\bv u_2^\top H_2(x)\bv u_2+ c\cdot \mathbb{I}\left(\bv u_2^\top \bv u_1\neq 0\right)$\\
	     $\cdots$ & $\cdots$ & $\cdots$\\
	     ``$k$'' & $\bv u_k\in\mathbb{S}^{d-1}$ & $\bv u_k^\top H_k(x)\bv u_k 
      + c\cdot\sum_{j<k}\mathbb{I}\left(\bv u_k^\top \bv u_j\neq 0\right)$\\
		\hline	
	\end{tabular}
    \end{center}
where  
\[
H_i(x) = H(x) - \sum_{j<i}\lambda_j(x)\bv u_j\bv u_j^\top
\] 
and $\mathbb{I}(\bv u_i^\top u_j \neq 0)$ is the indicator function or orthogonality condition:
\[
\begin{array}{cc}
     &\mathbb{I}(\bv u_i^\top u_j \neq 0) = 0 \text{ if } \bv u_i^\top \bv u_j = 0,  \\
     &\mathbb{I}(\bv u_i^\top u_j \neq 0) = 1 \text{ if } \bv u_i^\top \bv u_j \neq 0. 
\end{array}
\]
$c>0$ is a positive constant to    account for the orthogonality of $\{\bv u_j\}$ players at any optimal equilibrium of this game. 
For each action of player $1\le j\le k$,   $\bv u_j \in \mathbb{S}^{d-1}$, since the Hessian $H(\cdot)$ is bounded by Assumption \ref{asm}.c, there exist a positive constant $\bar{c}$ such that for any $c>\bar{c}$, the best response $\bv u_j$ of an arbitrary player $j\geq 1$ to any fixed other $k+1$ players $(x, y, \bv u_{-j})$ satisfies $ \bv u_i^\top\bv u_j = 0, \ \forall i<j.$.

Then we can extend the Theorem \ref{thm:main} so that we have the following theorem for the index-$k$ saddle point

\begin{thm}\label{thm:extend}
Suppose that {\bf Assumption} \ref{asmd} and {\bf Assumption} \ref{asm}  hold. There exists two positive constants $\bar{\rho}>0$ and $\bar{c}>0$ such that for any sufficiently large penalty factor $\rho>\bar{\rho}$ and $c>\bar{c}$, the following statements hold:
\begin{enumerate}
    \item If $x = x_k^*$ is a non-degenerate index-$k$ saddle point of the potential function $V$ and $\bv v_i(x)^\top \bv v_j(x) = 0$ for each $i\neq j$, then $(x,x,\bv v_{1:k}(x))$ is a Nash equilibrium of the game $G_{k+2}$;
    \item If $(x,y,\bv u_{1:k}(x))$ is a Nash equilibrium of the game $G_{k+2}$ and $y=x, \ \bv u_i(x) = \bv v_i(x)$ for each $i \in\{1,\cdots,k\}$, then $x = x_k^*$ is a non-degenerate index-$k$ saddle point of the potential function $V$.
    \end{enumerate}
\end{thm}

\begin{proof} 
    \begin{enumerate}       
        \item 
        To show that $(x,x,\bv v_1(x))$ is a Nash equilibrium, it suffices to show that each player minimizes their respective cost functions under this action profile, given the action of other players remains fixed. By expanding upon theorem \ref{lem:fs}, we have $\Phi_\rho(x_k^*) = \{x_k^*\}$. Consequently, we have
        \[
        x_k^* = \argmin_{y\in\mathbb{R}^d}\widetilde{W}_\rho(y; x_k^*, \bv v_{1:k}(x_k^*)),
        \]
        which implies that $x$ minimizes the cost function $\widetilde{W}_\rho(y;x,\bv u_{1:k})$ of player $-1$.
        Additionally, given that $\bv v_i(x)^\top \bv v_j(x) = 0$ for $i\neq j$, we have that for any $i \in\{1,\cdots,k\}$,
        \[
        \bv u_i = \bv v_i(x_k^*) = \argmin_{\bv u\in\mathbb{S}^{d-1}} \bv u^\top H_i(x_k^*)\bv u + c\cdot\sum_{j<i}\mathbb{I}(\bv u^\top\bv u_j\neq 0), 
        \]
        which implies that the cost functions of each player $i \in\{1,\cdots,k\}$ are minimized at $\bv u_i = \bv v_i(x^*_k)$. And the cost function $d(x,y)$ of player $0$ is minimized at $y=x$.
        Therefore, according to the definition of Nash equilibrium, if $x = x_k^*$, $(x,x,\bv v_{1:k}(x))$ is a Nash equilibrium of game $G_{k+2}$.
        \item With reference to the definition of Nash equilibrium, we establish that
        \[
        x = \argmin_{y\in\mathbb{R}^d} \widetilde{W}_\rho(y;x,\bv u_{1:k}) =
        \argmin_{y\in\mathbb{R}^d} \widetilde{W}_\rho(y;x,\bv v_{1:k}(x)).
        \]
        By extending theorem \ref{pro:convex}, we demonstrate that $\widetilde{W}_{\rho}(\cdot;x,\bv v_{1:k}(x))$ is strictly convex, which signifies that $\Phi_\rho(x) = {x}$. Consequently, we can show that
        \[
        \nabla_y\widetilde{W}_\rho(y;x,\bv v_{1:k}(x))|_{y=x} = 0
        \]
        and
        \[
        \nabla^2_y\widetilde{W}_\rho(y;x,\bv v_{1:k}(x))|_{y=x} \succ 0,
        \]
        leading to
        \[
        \nabla_x V(x) = 0 \text{ and } \lambda_1(x)\leq\cdots\leq\lambda_k(x)<0<\lambda_{k+1}(x)\leq\cdots\leq\lambda_d(x).
        \]
        Therefore, we conclude that $x = x_k^*$ is the index-$k$ saddle point of the potential function $V$.
        \end{enumerate}
\end{proof}

In addition, Algorithm \ref{alg:algorithm1} could be extended to index-$k$ saddle points easily as well. Compared with the index-1 saddle point, we need to solve the top $k$ eigenvectors of the Hessian matrix $H(x)$ and substitute the cost function with the function in equation \eqref{p_w} for each step of the iteration in Algorithm \ref{alg:algorithm1}, which has been studied in the literature \cite{yin2019high}.

\section{Numerical Results}\label{Num_ex} 
In this section, we will illustrate  our new IPM method, with a two-dimensional ODE toy model and an energy functional in a Hilbert space.
For all numerical tests, we have chosen the quartic function $d(x,y)=|x-y|^4 = \sum_{i=1}^d(x_i-y_i)^4$ as the penalty function $d$ if $x,y\in \Real^d$. For the energy functional, $d(\phi_1,\phi_2)=\|\phi_1-\phi_2\|^4$ with $\|\cdot\|$ is the norm in Hilbert space. The numerical comparison for the cubic and quartic penalty functions $d(x,y)$ is also shown in the Appendix.

 \subsection{ ODE Example}
 Consider the following two-dimensional potential function
 \cite{tpt2006}:
 \begin{equation}\label{energy}
\begin{split}
V(x,y) = 
&~3e^{-x^2} \left(  e^{- (y-\frac{1}{3})^2}- e^{-(y-\frac{5}{3})^2 } \right)
     -5e^{-y^2} \left ( e^{-(x-1)^2}  +  e^{-(x+1)^2} \right) 
     \\
     & +  \frac{1}{5}x^4 + \frac{1}{5}(y-\frac{1}{3})^4.
\end{split}
\end{equation}

The energy function \eqref{energy} exhibits three local minima approximately at $(1,0), (-1,0)$, and $(0,1.5)$ respectively, alongside a maximum at $(0,0.5)$. The energy function also features three saddle points exactly located at $(0.61727,1.10273)$, $(-0.61727,1.10273)$, and $(0, -0.31582)$.

In this experiment, we delve into the convergence properties of the iterative proximal minimization scheme with different penalty factors $\rho$. The original IMF-based methods \cite{IMA2015,IMF2015} correspond to $\rho=0$. The same inner iteration number  $M$ (refer to Algorithm \ref{alg:algorithm1} for more details)
is used in comparison.

Previous observations \cite{IMA2015} have revealed that utilizing a small value for $M$ typically boosts the robustness of the system, albeit at the cost of slow convergence rates. On the contrary, a larger  $M$ can harness the theoretical quadratic convergence rate, yet it may render the scheme susceptible to the initial guess. This can result in oscillation \cite{Ortner2016dimercycling} or even divergence.

 \begin{figure}[htbp]
	\centering
	\includegraphics[width=0.3\linewidth]{"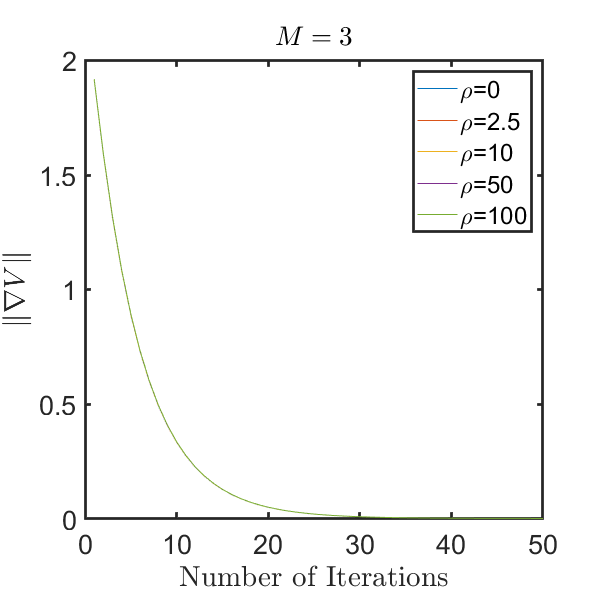"}
	\includegraphics[width=0.3\linewidth]{"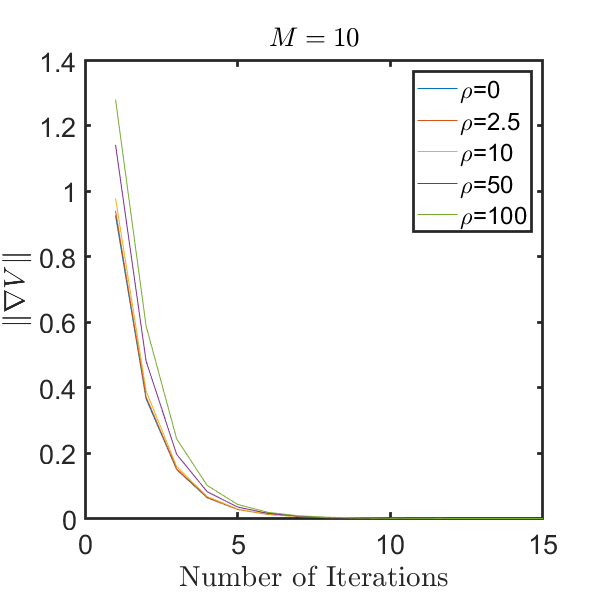"}
	\includegraphics[width=0.3\linewidth]{"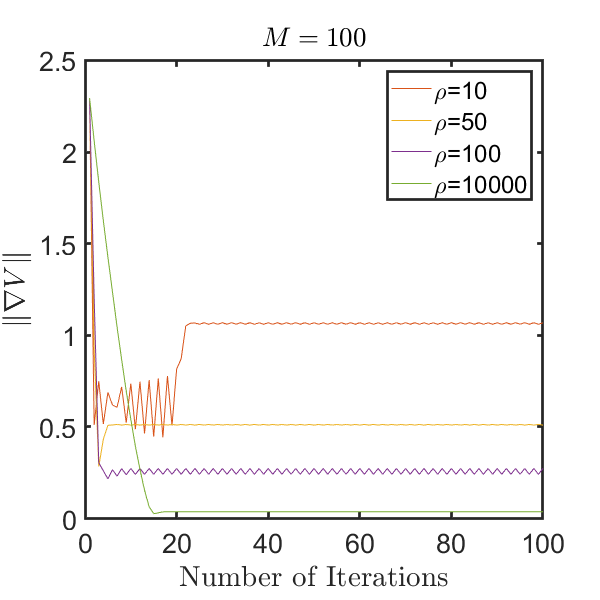"}
	\caption{ Decay of errors (measured by $\|\nabla V(x^t)\|$) in the   iterative proximal minimization   scheme, where $x$-axis is the number of iterations $k$.  $M$ is the steps of gradient descent in the subproblem for  $\widetilde{W}_\rho$ (see Algorithm \ref{alg:algorithm1}).}
	\label{fig:convergence_rate}
\end{figure}

\begin{figure}
	\centering
	\includegraphics[width=0.23\linewidth]{"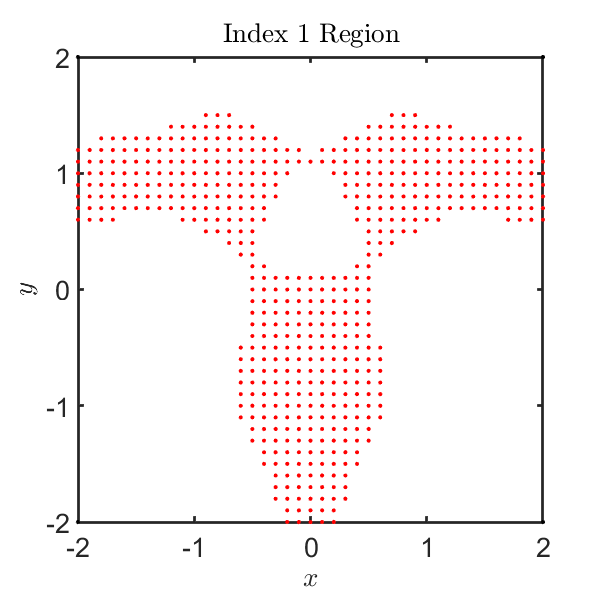"}
	\includegraphics[width=0.23\linewidth]{"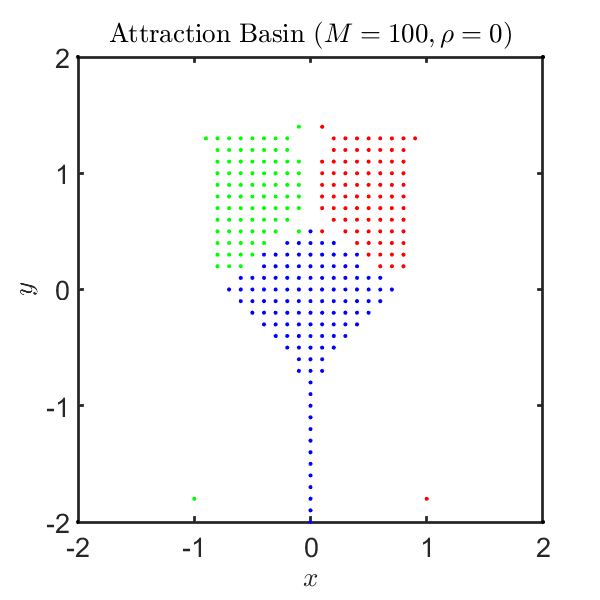"}
	\includegraphics[width=0.23\linewidth]{"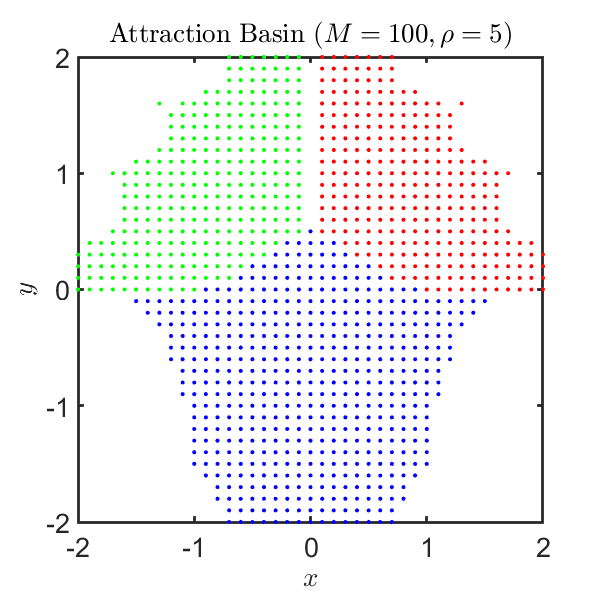"}
	\includegraphics[width=0.23\linewidth]{"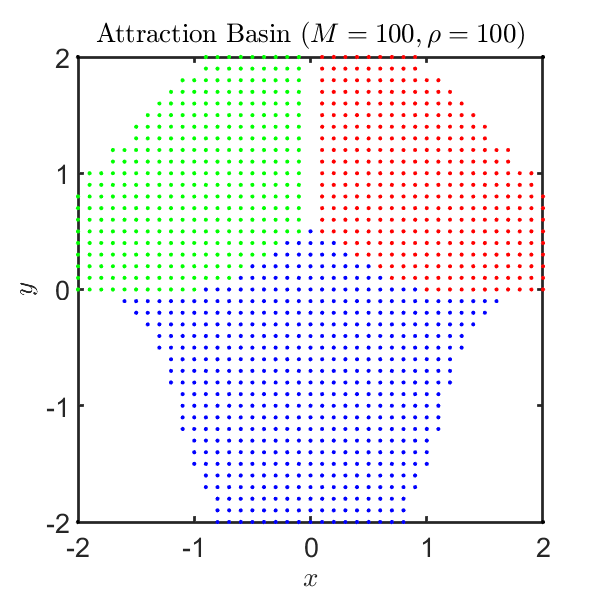"}
	\caption{Comparison of attraction basins towards each of three saddle points when $\rho$ varies. 
 Each color corresponds to the basin of each saddle point. 
 The index-1 region $\Omega_1$ of the function $V$ is shown in the first panel. }
	\label{fig:attraction_basin}
\end{figure}

In this example, we test different combinations of $\rho$ and $M$ for the IPM algorithm, starting from the initial point at $(x^0,y^0) = (0.5, 0.8)$, as depicted in the first two panels of Figure \ref{fig:convergence_rate}. We also test another initial point at $(x^0,y^0) = (1.5, 1.2)$, shown in the last panel of Figure \ref{fig:convergence_rate}.

Figure \ref{fig:convergence_rate} represents the decay of the error measured by $|\nabla V|$. For a small $M$ (as seen in the first and second panels of Figure \ref{fig:convergence_rate}), the minimization subproblem in each iteration is solved rather inexactly, and the penalty function $d(x,y)$ can scarcely take effect. Consequently, varying $\rho$ results in negligible differences in errors. However, if $M$ is excessively large, the effect of $\rho$ becomes instrumental in enhancing convergence.
In the last panel of Figure \ref{fig:convergence_rate}, we deliberately select an unfavorable initial point at $(x^0,y^0) = (1.5, 1.2)$. Here, $\rho=0$ leads to divergence, but as we increase $\rho$, we observe that the iteration can converge to the desired saddle point after exhibiting some oscillatory behavior at an intermediate size of $\rho$.

We further examine the convergence behaviors under different values of $\rho$ by illustrating the basin of attraction of our IPM scheme at $M=100$ in Figure \ref{fig:attraction_basin}. 
This figure represents the collection of initial points converging to one of the three saddle points. At $\rho=0$, it has been demonstrated in \cite{IMA2015} that the basins of attraction of the three saddle points shrink as $M$ increases. The second panel in Figure \ref{fig:attraction_basin} displays this small basin of attraction at $\rho=0$ and $M=100$. For the same value of $M=100$, we observe that the attraction basin considerably expands at $\rho=5$. When $\rho=100$, the basin of attraction completely covers the whole index-1 region $\Omega_1$, thereby numerically validating Theorem \ref{pro:convex}.

 \subsection{Cahn-Hilliard equation}
 The second example is the Cahn-Hilliard equation \cite{CH-EQ} which has been widely used in many   moving interface problems in material sciences and fluid dynamics through a phase-field approach. 
  Consider the Ginzburg-Landau free energy on a one-dimensional interval $[0,1]$ 
  \begin{equation}
 \label{eqn:F_GL0}
F(\phi) = \int_\Omega \left[ \frac{\kappa^2}{2}
|\nabla \phi(x)|^2+ f(\phi(x)) \right]\,dx, \quad f(\phi) = (\phi^2-1)^2/4,
\end{equation}
with $\kappa=0.04$ and the constant mass $\int \phi dx=0.6$.
 The 
 Cahn-Hilliard (CH) equation 
 is the $H^{-1}$-gradient flow of $F(\phi)$,
\begin{equation}\label{CH_eq}
\frac{\partial\phi}{\partial t} = \Delta\frac{\delta F}{\delta\phi}=-\kappa^2 \Delta^2 \phi + \Delta (\phi^3-\phi).
\end{equation}
Here $\frac{\delta F}{\delta \phi}$ is the first order variation of $F$
in the standard $L^2$ sense. 

We are interested in the transition state of the Cahn-Hilliard equation, which is the index-1 saddle point of Ginzburg-Landau free energy in the $H^{-1}$ Riemannian metric.
In the IPM method, we take the quartic proximal function, so the auxiliary functional in Algorithm \ref{alg:algorithm1} then becomes
\begin{align}
    \widetilde{W}_\rho(\phi;\phi^{l}) = \int_\Omega \Big[ \frac{\kappa^2}{2}|\nabla\phi|^2 + f(\phi) - \kappa^2 |\nabla\hat\phi|^2 - 2f(\hat\phi) \Big]\,dx + \rho\int_\Omega |\phi-\phi^{l}|^4 \,dx.\label{auxiliary_CH}
\end{align}
The gradient flow of $\widetilde{W}_\rho(\phi,\phi^{l})$ in $H^{-1}$ metric is
$   \frac{\partial\phi}{\partial t} = \Delta \delta_\phi \widetilde{W}_\rho(\phi),
$
with $
 \delta_\phi \widetilde{W}_\rho(\phi) = -\kappa^2 \Delta \phi + (\phi^3-\phi) + 2 \inpd{\bv v_1}{ \kappa^2 \Delta \hat\phi - (\hat\phi^3 - \hat\phi) }_{L^2} \bv v_1 + 4\rho(\phi-\phi^{l})^3, $ and $\hat\phi=\phi^{l}+\inpd{\bv v_1}{\phi-\phi^{l}}_{H^{-1}}\bv v_1$.

In the numerical simulation, we use the uniform
  mesh grid for spatial discretization  $\{x_i = i h, i=0, 1, 2,\ldots, N\}. ~h = 1/N.$ $N=100$ and the time step is to solve the gradient flow for $\wt{W}$ is $\Delta t = 0.1$. The periodic boundary condition is considered.
To test  the IPM, we  first choose the proximal constant  $\rho=100$.
The saddle point of $F(\phi)$ is reproduced in Figure \ref{fig:saddle}, which is the same as in  \cite{convex_IMF, ProjIMF}.
Besides, the quadratic convergence rate of the IPM algorithm is also verified empirically; see Figure \ref{fig:Modif_IMF_convergence_rate} where ``Iteration'' means the outer iteration (cycle)  index $t$ in Algorithm \ref{alg:algorithm1}.
  To illustrate the advantage of this method, we make a comparison of the convergence tests between the original IMF ($\rho=0$) and the proximal method ($\rho=100$), by using  different initial states and  the different inner iteration number $M$.  
Table \ref{1D_case} shows the convergence/divergence results for the four initial states $\phi_{01}$   to $\phi_{04}$ (shown in   Figure \ref{fig:saddle}).  We have the same conclusion as in the previous example that the introduction of proximal penalty in the IPM can stabilize the convergence for a wider range of $M$ than the original IMF.

\begin{table}[htbp]
{
\footnotesize   \caption{Comparison of numerical convergence from three differential initial guesses $\phi_{01}$,$\phi_{02}$ and $\phi_{03}$ shown in Figure \ref{fig:saddle}. ``IMF" means the original IMF ($\rho=0$); ``IPM" is the 
IPM with $\rho=100$.
$M$ is the number of gradient descent steps in minimizing the auxiliary functions. 
``$\checkmark$" and ``\ding{55}" mean the convergent and divergent results, respectively. }
\label{1D_case}
\begin{center}
    \begin{tabular}{|c|r|r|r|r|r|r|r|r|}
        \hline
        \multirow{2}{*}{$M$} & \multicolumn{2}{c|}{$\phi_{01}$} & \multicolumn{2}{c|}{$\phi_{02}$} & \multicolumn{2}{c|}{$\phi_{03}$} & \multicolumn{2}{c|}{$\phi_{04}$} \\ \cline{2-9}
         &  IMF  & IPM  &  IMF  & IPM &  IMF  & IPM  &  IMF  & IPM \\ \hline
         10 & $\checkmark$ & $\checkmark$ & $\checkmark$ & $\checkmark$ & $\checkmark$ & $\checkmark$ & $\checkmark$ & $\checkmark$ \\ \hline 
          100 & $\checkmark$ & $\checkmark$ & $\checkmark$ & $\checkmark$ & \ding{55} & $\checkmark$ & $\checkmark$ & $\checkmark$ \\ \hline 
           200 & $\checkmark$ & $\checkmark$ & \ding{55} & $\checkmark$ & \ding{55} & $\checkmark$  & \ding{55} & $\checkmark$ \\ \hline 
            500 & \ding{55} & $\checkmark$ & \ding{55} & $\checkmark$ & \ding{55} & $\checkmark$  & \ding{55} & $\checkmark$ \\ \hline 
    \end{tabular}
  \end{center}
 }
\end{table}

\subsection{Allen-Cahn equation}
 The  Ginzburg-Landau free energy \eqref{eqn:F_GL0} gives arise to the Allen-Cahn equation in the  $L^2$ metric,
\begin{equation}\label{AC_eq}
\frac{\partial\phi}{\partial t} = -\frac{\delta F}{\delta\phi}= \kappa^2 \Delta \phi - (\phi^3-\phi).
\end{equation} 
We use  this example to illustrate more details about  the performance of the IPM method
for calculating the index-1 saddle solution of $F$ in $L^2$ metric.
The auxiliary functional $\widetilde{W}_\rho$ is the same as that in \eqref{auxiliary_CH}, with  the $L^2$-gradient flow of $\widetilde{W}_\rho(\phi,\phi^{l})$ in the form
\begin{align*}
\frac{\partial\phi}{\partial t} & =  - \delta_\phi \widetilde{W}_\rho(\phi) \\
& = \kappa^2 \Delta \phi - (\phi^3-\phi) - 2 \inpd{\bv v_1}{ \kappa^2 \Delta \hat\phi - (\hat\phi^3 - \hat\phi) }_{L^2} \bv v_1 - 4\rho(\phi-\phi^{l})   \abs{\phi-\phi^l}^2,
\end{align*}
with 
 $\hat\phi=\phi^{l}+\inpd{\bv v_1}{\phi-\phi^{l}}_{L^2} \bv v_1$.
In the numerical tests, we take $\kappa=0.1$ and  periodic boundary condition.

Figure \ref{fig:saddle_AC} shows the  transition state (solid line) of the Allen-Cahn equation and the  initial condition (dashed line) used in all tests. 
The results shown in Figure \ref{fig:force_k_M_600},\ref{fig:force_k_rho_03},\ref{fig:force_k_M_10000} and \ref{fig:force_k_rho_10}  describe the decay of the error measured by $\| \delta_\phi F(\phi^{l})\|_{L^2}$, in various settings, with respect to the number of the outer cycles (``Iteration'' in the figures,  referring to iteration index $t$ in Algorithm \ref{alg:algorithm1}.)
By testing many cases for various $\rho$ and $M$, we have the   observation 
that when $\rho=0$ or too small, the algorithm is prone to diverge if the inner iteration number $M$  is large.  
Quantitatively, by tuning up the value of $\rho$,
we observe the increasing size of  the range for the valid $M$ such that the algorithm converges. For example, when $\rho=0.3$, $M$ can be chosen as large as  nearly $1272$ to safely run the algorithm (Figure \ref{fig:force_k_rho_03}),
and as large as $10000$ if  $\rho=1$ (Figure \ref{fig:force_k_M_10000}); but 
at $\rho=0$, this valid range for $M$ is only $[1,450]$ (Figure \ref{fig:force_k_rho_10}).

We also   explored  the efficiency issue by measuring the error decay in terms of the total cost.
Table \ref{tab:effect_of_M} suggests that the efficiency is not significantly affected by the inner iteration $M$. However, the value of $\rho$  may have a significant effect.
Figure \ref{fig:force_k_M_600} at $M=600$ and Figure \ref{fig:force_k_M_10000} at $M=10000$ demonstrate that a smaller $\rho $ can offer a better decay of the error, in particular for the very large $M$. These empirical findings imply that robustness and efficiency typically trade off: the former requires a large $\rho$, while the latter is preferred to a small $\rho$. The ideal $\rho$ should be fine-tuned to find a balance between them.

\begin{figure}[htbp]
\centering
\subfloat[Transition state]{\label{fig:saddle_AC}\includegraphics[scale=0.30]{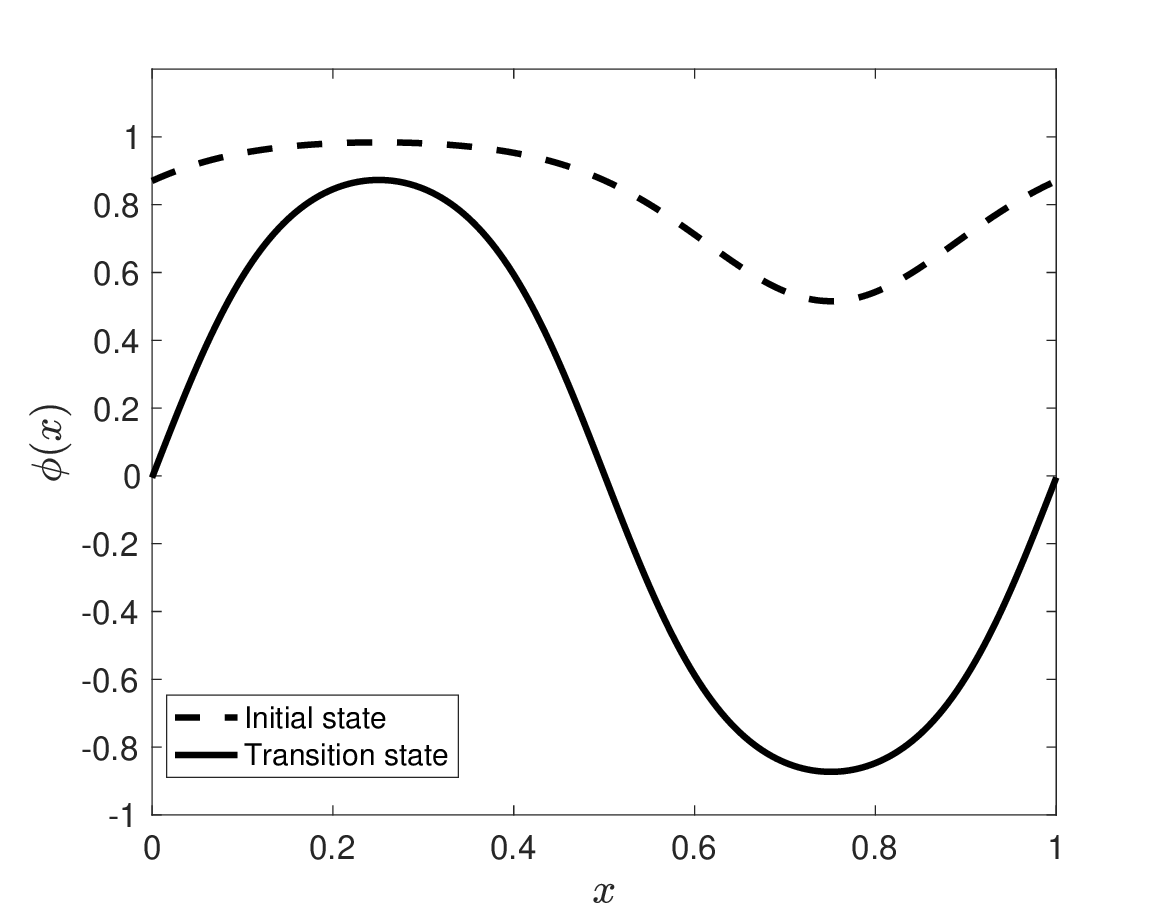}}

\subfloat[Decay of error (fixed $M=600$)]{\label{fig:force_k_M_600}\includegraphics[width=0.48\textwidth]{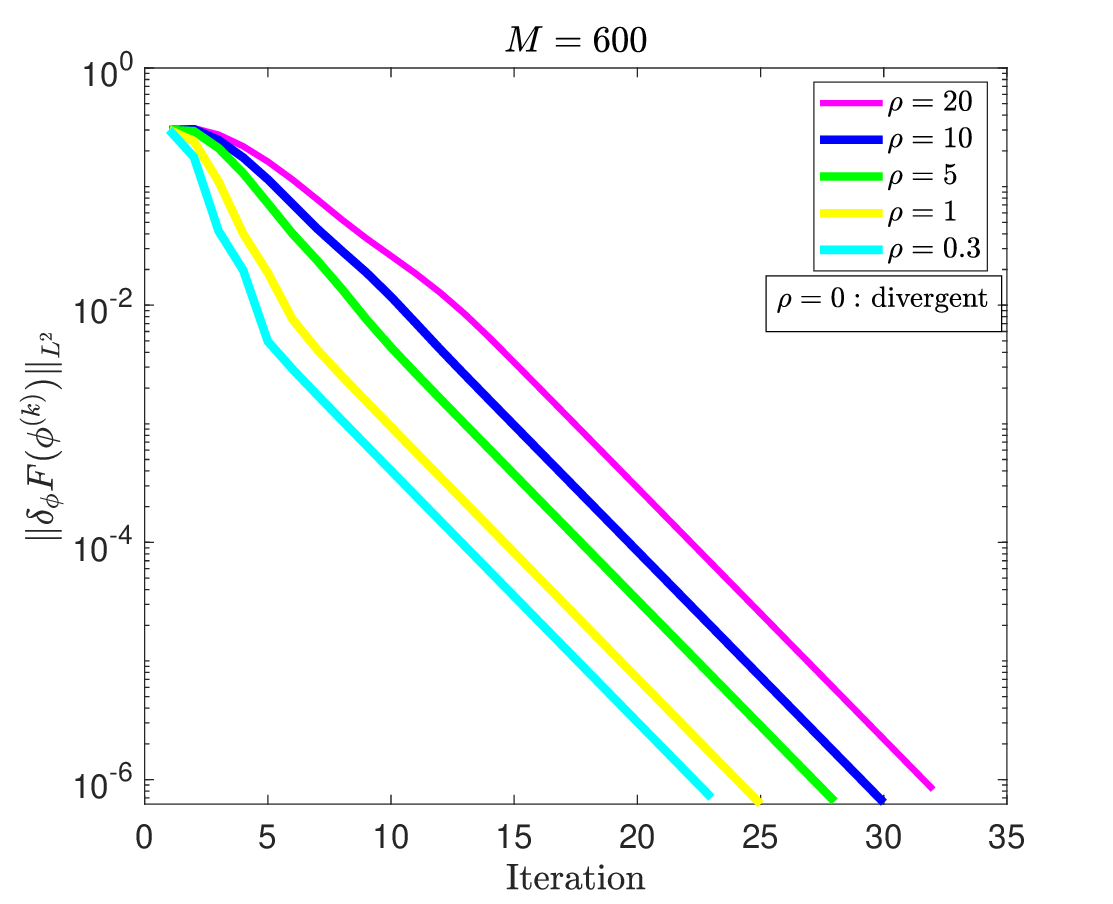}}
\subfloat[Decay of error]{\label{fig:force_k_rho_03}\includegraphics[width=0.48\textwidth]{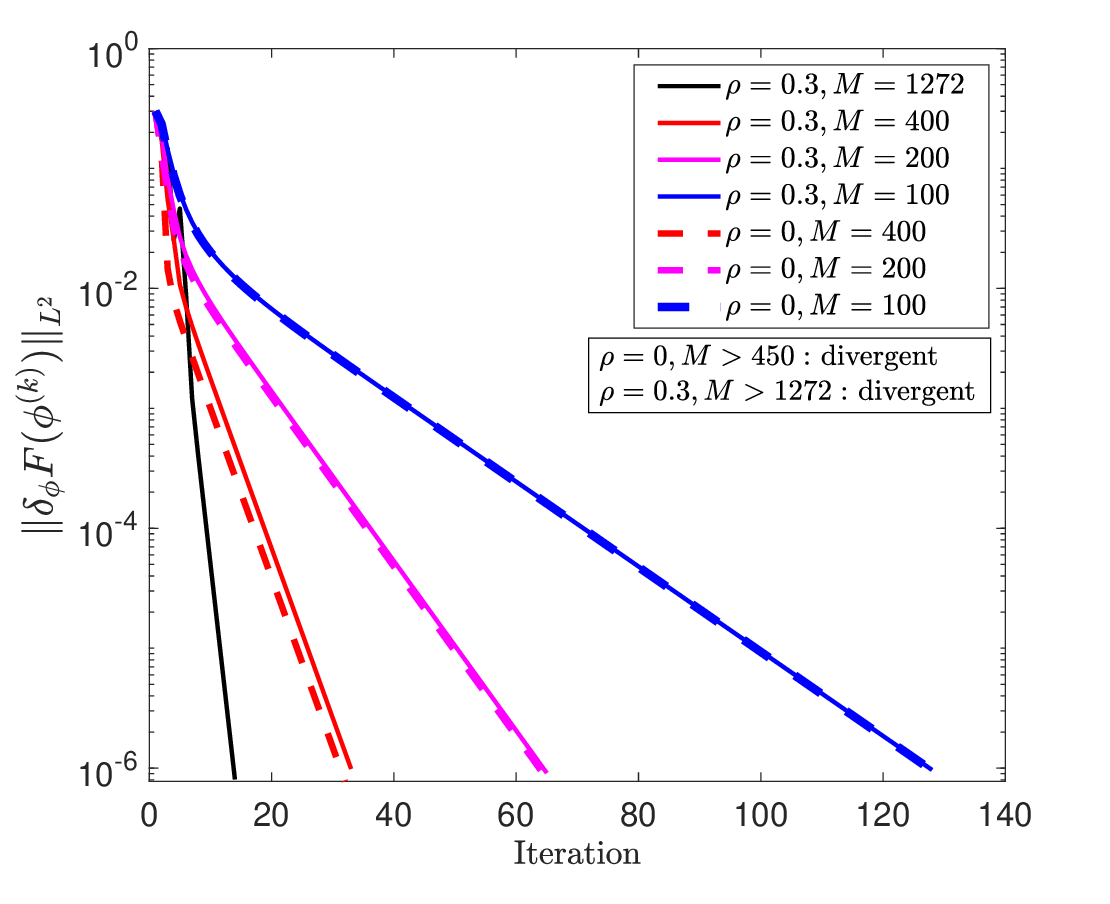}}
\\
\subfloat[Decay of error (fixed $M=10000$)]
{\label{fig:force_k_M_10000}\includegraphics[width=0.48\textwidth]{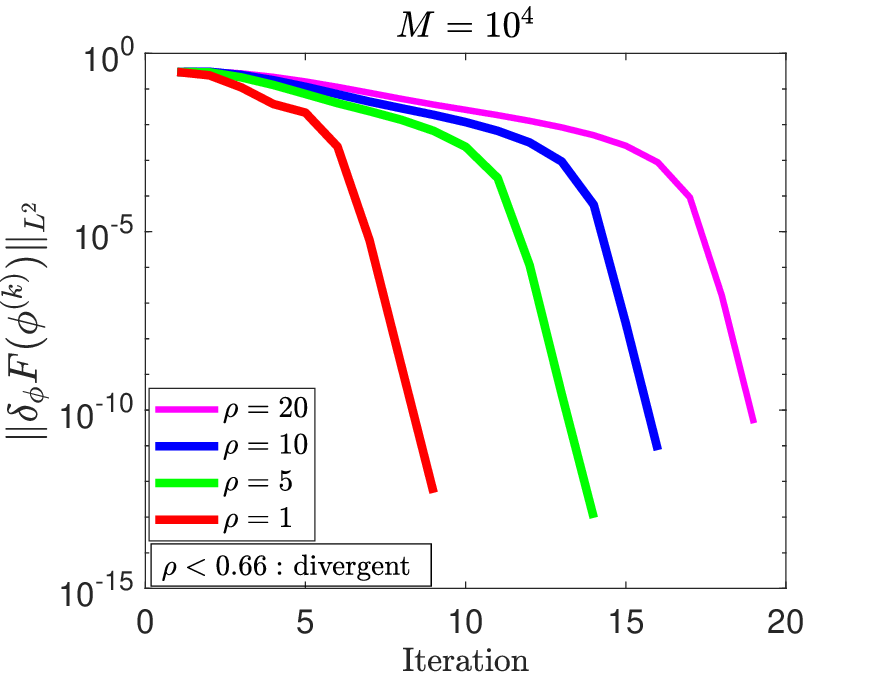}}
\subfloat[Decay of error]{\label{fig:force_k_rho_10}\includegraphics[width=0.48\textwidth]{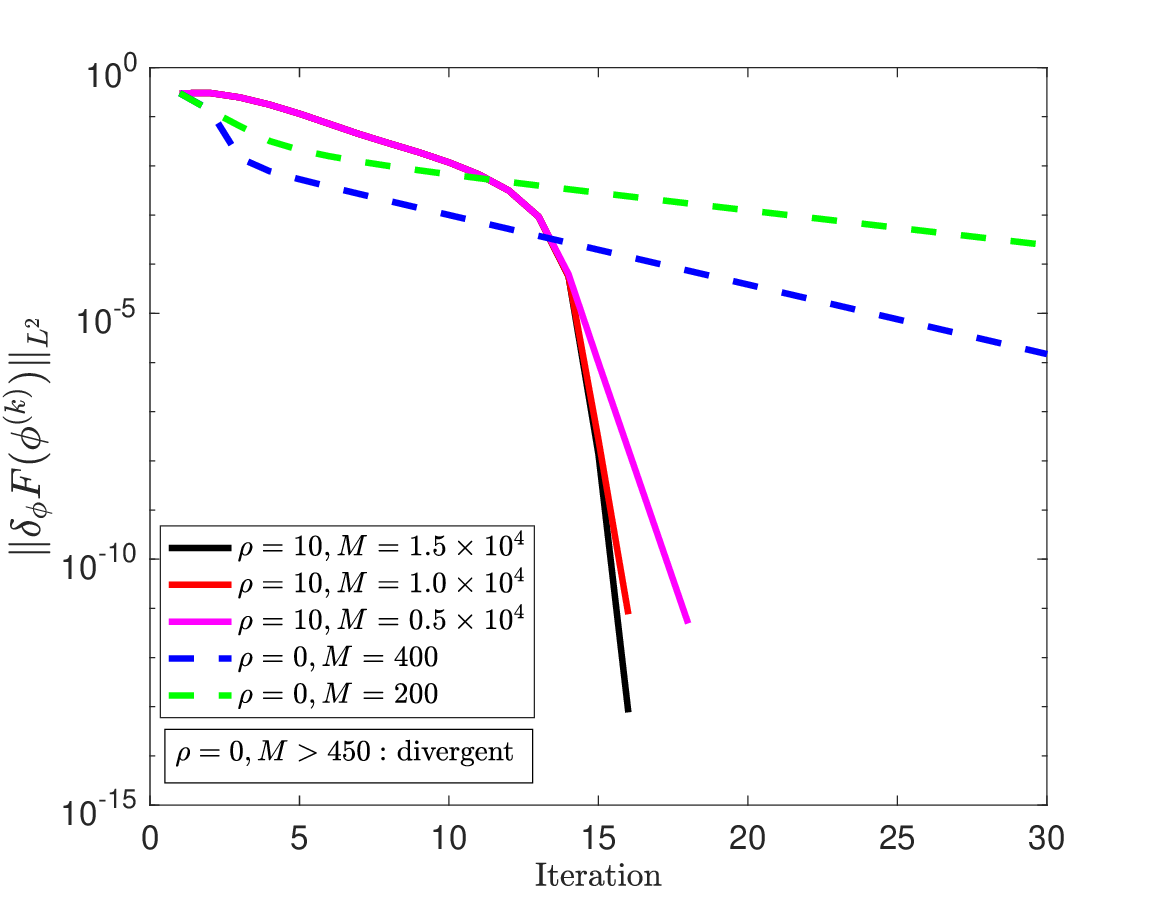}}
\caption{ Results by adding the penalty term: $\rho |\phi -\phi^l|^4 (b=4):$  (a) Transition state (solid curves) and the initial state (dashed line);
(b) The decay of the error $\| \delta_\phi F(\phi^{l})\|_{L^2}$ with the outer cycle for fixed inner iteration number  $M=600$ by using various $\rho = 0.3, 1, 5, 10, 20.$
 (c) The decay of the error $\| \delta_\phi F(\phi^{l})\|_{L^2}$ with the outer cycle when $\rho=0.3$(solid lines) and $\rho=0$ (dashed lines), respectively,  by using various inner iteration number 100, 200, 400 and 1272. The case for $M > 450$ with $\rho=0$ is divergent; (d) The decay of the error $\| \delta_\phi F(\phi^{l})\|_{L^2}$ with the outer cycle for fixed inner iteration number  $M=10^4$ by using various $\rho = 1, 5, 10, 20.$
 (e) The decay of the error $\| \delta_\phi F(\phi^{l})\|_{L^2}$ with the outer cycle when $\rho=10$(solid lines) and $\rho=0$ (dashed lines), respectively,  by using various inner iteration number $0.5\times 10^4, 1.0\times 10^4 $ and $ 1.5\times 10^4 $. The case for $M > 450 $ with $\rho=0$ is divergent.  }
 \label{saddle_AC}
\end{figure}

\begin{table}[htbp]
\footnotesize
 \caption{The effect of inner iteration number $M$. (a) and (b) show the required number of outer cycle $Iter$ and the total cost as the multiplication of $M$ and $Iter$, when reach the   error $10^{-6}$ for various inner iteration number $M=100, 200, 400$ by using  different $\rho=0 $ and $\rho=0.3$, respectively. The required number of outer cycle decreases with the inner iteration number $M$ increases, while the total cost almost keeps unchanged. This result is obtained with the initial state shown as the dashed line in Figure  \ref{fig:saddle_AC}.}
    \label{tab:effect_of_M}
    \centering
    \begin{tabular}{ccc}
         \hline
           $M$  &  Iter & total cost \\
           \hline
        100 & 128 & 12800  \\ 
        200 & 64 & 12800 \\ 
        400 & 32 & 12800 \\   
         \hline
    \end{tabular}
    \hspace{4mm}
    \begin{tabular}{ccc}
        \hline
        $M$  &  Iter & total cost \\
        \hline
        100 & 127 & 12700  \\ 
        200 & 64 & 12800 \\ 
        400 & 32 & 12800 \\  
        \hline
    \end{tabular}\\
    \vspace{2mm}
    (a) $\rho=0 $ \hspace{2.5cm} (b) $\rho=0.3$ 
\end{table}

 \subsection{ Comparison of the cubic $(b=3)$ and quartic $(b=4)$ penalty function $d(x,y)=|x-y|^b$}
In this last part, we briefly investigate these two choices of the  proximal penalty function by revisiting the above numerical examples.
For the ODE example, Figure \ref{fig:attraction_basin3} shows that for the same value of $\rho$,  a larger attraction basin will be obtained with the cubic penalty function $(b=3)$, compared with the quartic penalty function $(b=4)$ shown in Figure \ref{fig:attraction_basin}.

\begin{figure}[htbp]
	\centering
    \includegraphics[width=0.3\linewidth]{"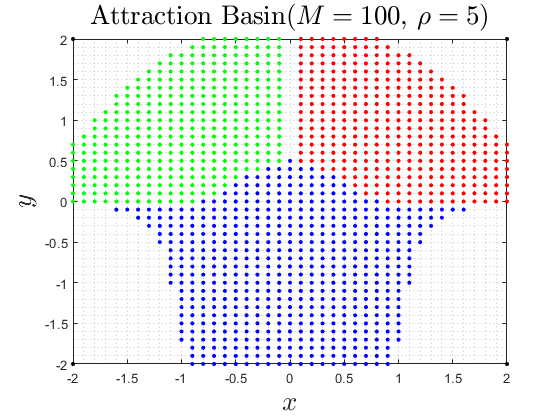"}
	\includegraphics[width=0.3\linewidth]{"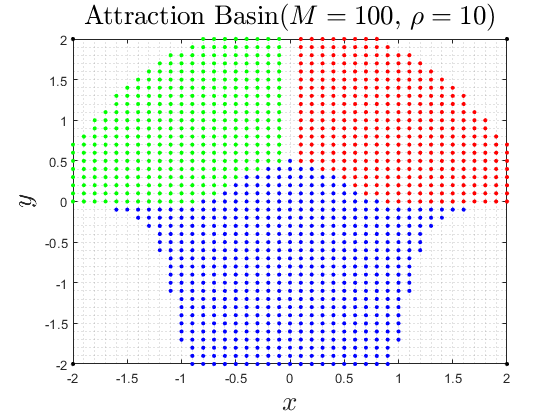"}
        \includegraphics[width=0.3\linewidth]{"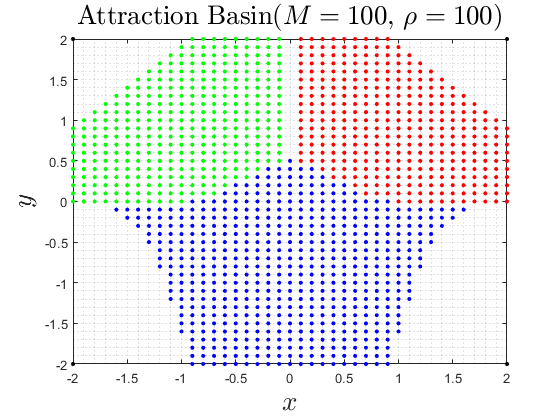"}
	\caption{Comparison of attraction basins towards each of three saddle points when $\rho$ varies, for a cubic penalty function $d(x,y) = \sum_{i=1}^d\|x_i-y_i\|^3$. With the same value of $\rho$, we have a larger attraction basin with the cubic penalty function, compared with quartic penalty function (Fig. \ref{fig:attraction_basin}).}
	\label{fig:attraction_basin3}
\end{figure}

For the Allen-Cahn equation, Figure \ref{Comparison_b3_4} shows the numerical results when using the cubic penalty function $(b=3)$.   Figure \ref{fig:force_k_M1e4_b3} for $d=3$  is consistent with Figure \ref{fig:force_k_M_10000} ($d=4$). 
Figure \ref{fig:force_k_M1e4_b_3_4} shows both results for $d=3$ (solid lines) and $d=4$ (dashed lines), highlighting two key observations: (1) for fixed $M$ and $\rho$, achieving the same tolerance requires significantly more iterations for $b=3$ compared to $b=4$, with the gap widening as  $\rho$ increases. (2) On the other hand, for the same large inner iteration number $M=10^4$, $b=3$ accommodates a smaller $\rho$ (e.g., $\rho=0.4$), whereas for $b=4$, the algorithm diverges when $\rho<0.66$.  Figure \ref{fig:force_k_rho_04_b_3_4} has the comparisons with a   small $\rho=0.4$. 
For a smaller inner iteration $M=10^3$, both cases exhibit nearly identical linear convergence rate. However,  $b=4$ can not support a sufficiently large $M$. In contrast, $b=3$ can support an inner iteration number as high as $M=2\times 10^4$.

In conclusion,   $b=3$ in general implies a stronger effect of stabilization for the proximal penalty, enhancing the algorithm's robustness, at the cost of slow convergence.

\begin{figure}[htbp!]
\centering
\subfloat[Decay of error]{\label{fig:force_k_M1e4_b3}\includegraphics[scale=0.30]{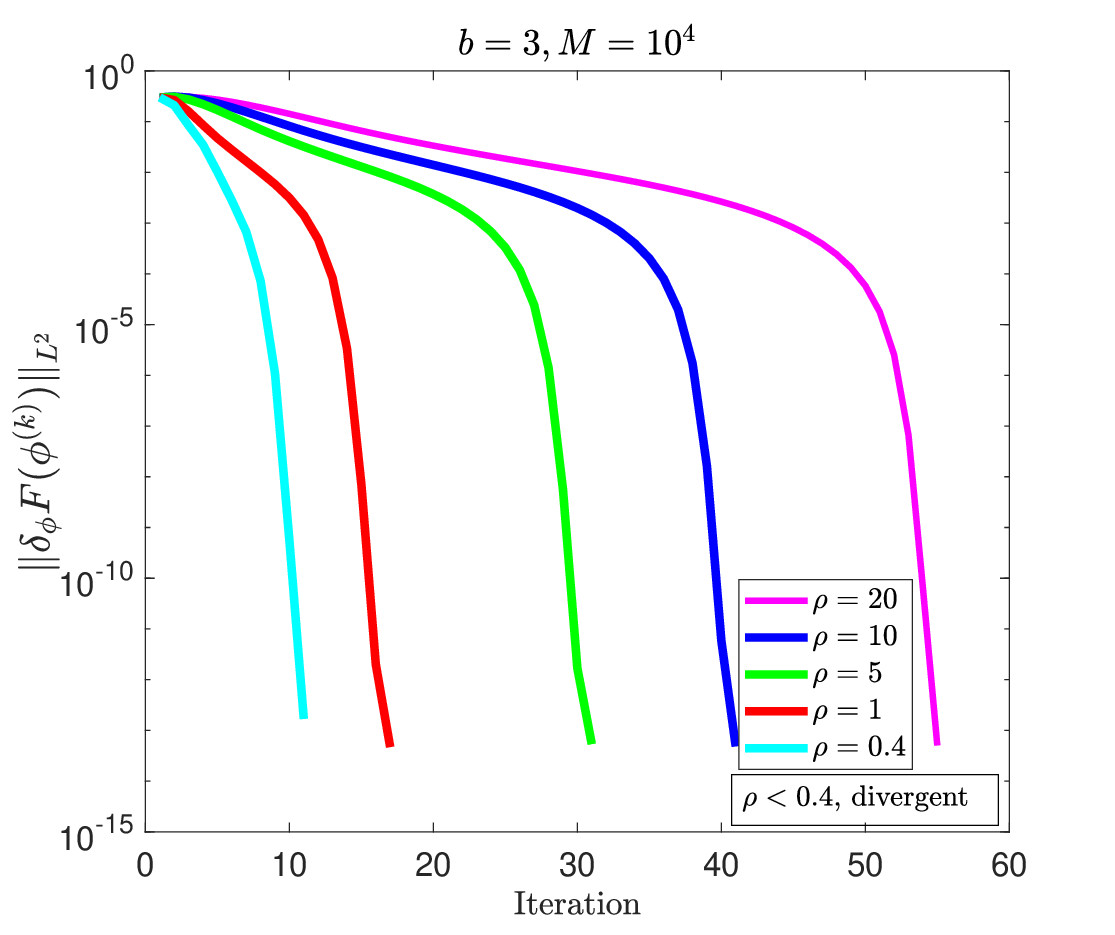}}

\subfloat[Decay of error]{\label{fig:force_k_M1e4_b_3_4}\includegraphics[width=0.48\textwidth]{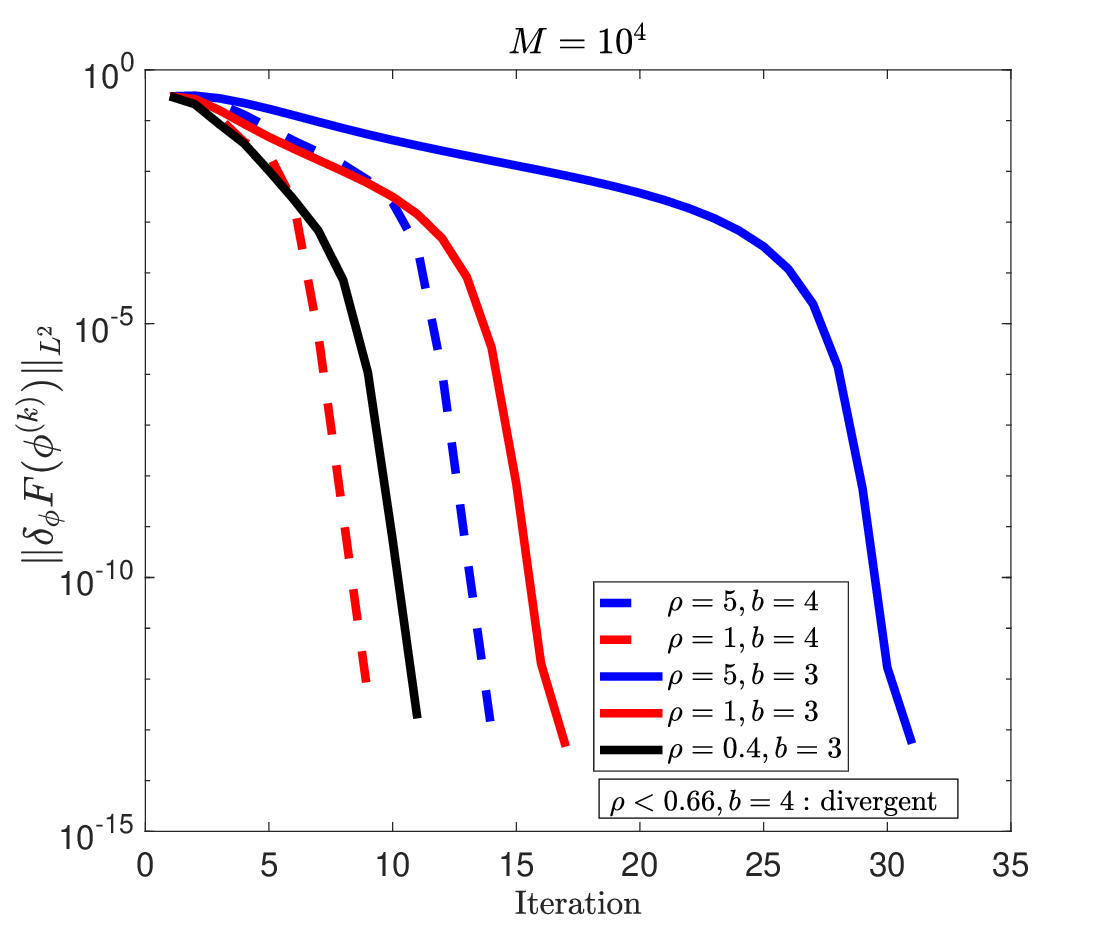}}
\subfloat[Decay of error]{\label{fig:force_k_rho_04_b_3_4}\includegraphics[width=0.48\textwidth]{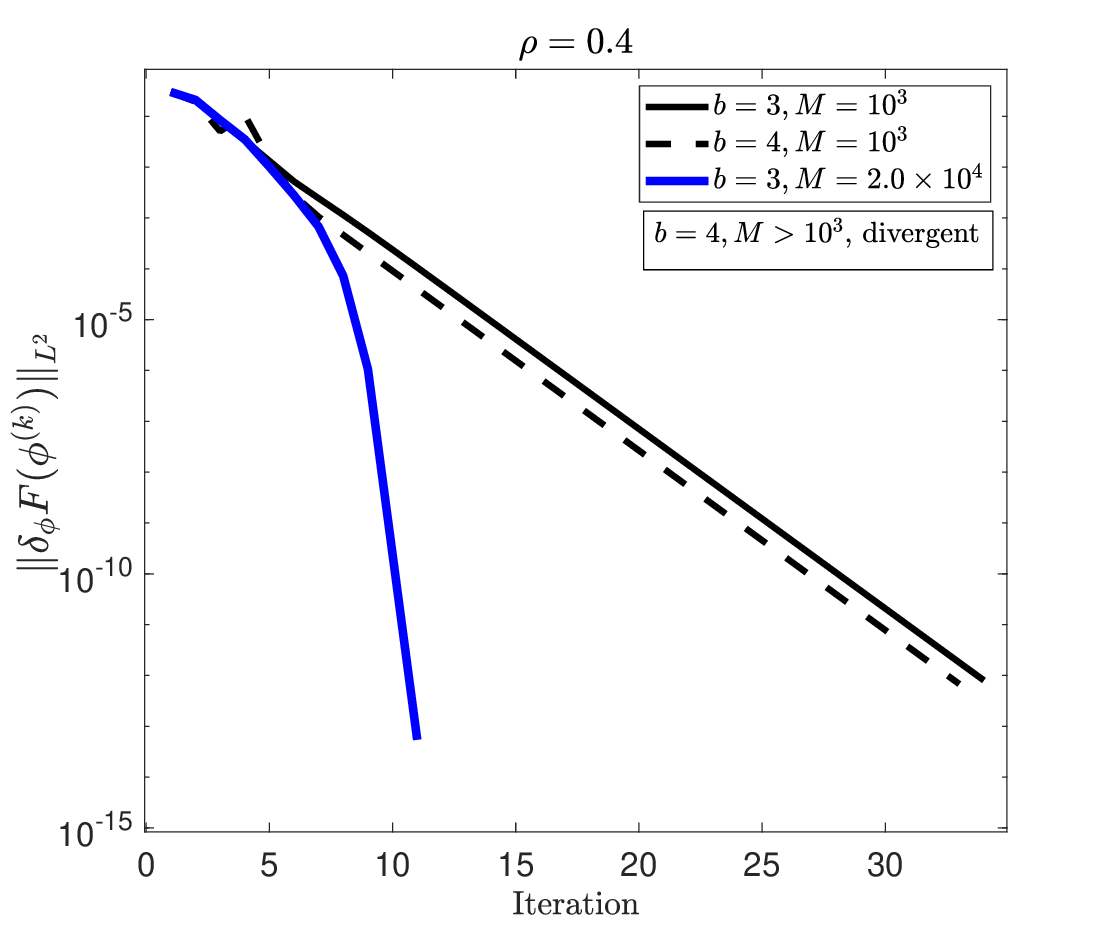}}
\caption{(a) The decay of the error $\| \delta_\phi F(\phi^{l})\|_{L^2}$ with the outer cycle for fixed inner iteration number $M=10^4$ and with the penalty $\rho |\phi-\phi^l|^3$ (i.e., b=3), by using various $\rho = 0.4, 1, 5, 10, 20.$
(b) The comparison of the error decay for $b=3$ (solid lines) and $b=4$ (dashed lines) when using various $\rho=0.4, 1, 5$. Here, the inner iteration number is fixed at $M=10^4$. The case for $b=4$ when $\rho<0.66$ is divergent.
(c) The  comparison of the error decay for $b=3$  (solid lines) and $b=4$  (dashed lines) when using various inner iteration number $M=10^3$ and $2 \times 10^4$. Here $\rho=0.4$ is fixed and the case for $b=4$ with $M > 10^3$ is divergent. }
 \label{Comparison_b3_4}
\end{figure}

  \section{Conclusion} \label{con}

We have introduced the Iterative Proximal Minimization (IPM) algorithm to calculate the index-1 saddle points of the potential function by casting the algorithm in the perspective of Nash equilibrium for the differential game.
The   numerical contributions of our work include notable improvements in practical robustness as well as a theoretical guarantee of convergence starting from any index-1 region. Adding a proximal penalty function that increases faster than quadratics strengthens this robustness. Using the same proximal function $\rho\cdot d(x,y)$, the extension to saddle points of index-$k$ is also achieved. The new algorithm
provides a more reliable methodology for preparing initial guesses and setting parameters in bilevel optimization problems for  the iterative minimization formulation\cite{IMF2015}. However, our numerical tests also indicate the slowdown in efficiency if the proximal term is too strong, suggesting
a decaying schedule of the proximal constant $\rho$ for practical applications.
  

 \bibliography{gad,my,ms}
 
\bibliographystyle{siam} 

\end{document}

%% file: ex_shared.tex
 \usepackage{amsfonts}
\usepackage{graphicx}
\usepackage{epstopdf}
\usepackage{algorithmic}
\ifpdf
  \DeclareGraphicsExtensions{.eps,.pdf,.png,.jpg}
\else
  \DeclareGraphicsExtensions{.eps}
\fi

\newsiamremark{remark}{Remark}
\newsiamremark{hypothesis}{Hypothesis}
\crefname{hypothesis}{Hypothesis}{Hypotheses}
\newsiamthm{claim}{Claim}

\headers{ Iterative Proximal-Minimization for Saddle Points}{S. Gu, H. Zhang, X. Zhang and X. Zhou}

\title{
 Iterative Proximal Minimization for Computing Saddle Points with Fixed index\thanks{Submitted to the editors DATE.
\funding{ Shuting Gu is supported  by  NSFC 11901211 and the Natural Science Foundation of Top Talent of SZTU GDRC202137. 
  Xiang Zhou acknowledges the partial support of Hong Kong RGC GRF  11308121, 11318522, 11308323.   }
  }
  }

\author{Shuting Gu\thanks{College of Big Data and Internet, Shenzhen Technology University, Shenzhen 518118, P.R. China 
  (\email{gushuting@sztu.edu.cn}).}
  \and Hao Zhang\thanks{Department of Data Science, City University of Hong Kong,
Hong Kong SAR.}
 \and Xiaoqun Zhang\thanks{Institute of Natural Science, Shanghai Jiaotong University, Shanghai, 
China (\email{xqzhang@sjtu.edu.cn}.)}
  \and Xiang Zhou\footnotemark[3]\thanks{   Department of Mathematics,
City University of Hong Kong, Hong Kong SAR (\email{xizhou@cityu.edu.hk}). } 
}

\usepackage{amsopn}

\makeatletter
\newcommand*{\addFileDependency}[1]{
  \typeout{(#1)}
  \@addtofilelist{#1}
  \IfFileExists{#1}{}{\typeout{No file #1.}}
}
\makeatother

\newcommand*{\myexternaldocument}[1]{%
    \externaldocument{#1}%
    \addFileDependency{#1.tex}%
    \addFileDependency{#1.aux}%
}